\documentclass[reqno,10pt]{amsart}
\usepackage{amsmath,amssymb}
\usepackage[mathscr]{euscript}
\usepackage{xcolor}

\theoremstyle{plain}
\newtheorem{theorem}{Theorem}[section]

\newtheorem{lemma}[theorem]{Lemma}
\newtheorem{corollary}[theorem]{Corollary}

\theoremstyle{definition}
\newtheorem{definition}[theorem]{Definition}

\theoremstyle{remark}

\numberwithin{equation}{section}
\numberwithin{theorem}{section}
\newcommand{\mc}[1]{{\mathcal #1}}
\newcommand{\mf}[1]{{\mathfrak #1}}

\newcommand{\ms}[1]{{\mathscr #1}}
\newcommand{\bb}[1]{{\mathbb #1}}

\DeclareMathOperator{\tr}{Tr}

\DeclareMathOperator{\Ent}{Ent}

\renewcommand{\epsilon}{\varepsilon}
\renewcommand{\tilde}{\widetilde}
\makeatletter
\newcommand{\leqnomode}{\tagsleft@true\let\veqno\@@leqno}
\makeatother

\title[MFT and Schr\"odinger problems]
{
Macroscopic fluctuation theory \\ from a Lagrangian viewpoint \\ and the Schr\"odinger problem
}

\author [L.\ Bertini]{Lorenzo Bertini}
\address{Lorenzo Bertini \hfill\break \indent
  Dipartimento di Matematica, Universit\`a di Roma `La Sapienza',
  \hfill\break \indent
  I-00185 Roma, Italy\hfill\break\vskip 1pt
  Rapha\"el Chetrite \hfill\break \indent
  CNRS Laboratoire Ypatia des Sciences Math\'ematiques (LYSM)
  \hfill\break \indent
  I-00185 Roma, Italy}

\author [D.\ Gabrielli]{Davide Gabrielli}
\address{Davide Gabbrielli \hfill\break \indent
	DISIM, Universit\`a dell' Aquila,
	\hfill\break \indent
	Via Vetoio Loc. Coppito 
	\hfill\break \indent
	67100 L'Aquila, Italy}

\author [G.\ Jona-Lasinio]{Giovanni Jona-Lasinio}
\address{Giovanni Jona-Lasinio \hfill\break \indent
	Dipartimento di Fisica and INFN, Universit\`a di Roma `La Sapienza',
	\hfill\break \indent
	I-00185 Roma, Italy}

\dedicatory{
  \textbf{With an Appendix by Rapha\"el Chetrite}\\
  \vskip 12pt
  This paper is dedicated to the memory of our friend
  Giuseppe Da Prato
}


\begin{document}

\begin{abstract}
  We formulate the Schr\"odinger problem
  for interacting particle
  systems in the hydrodynamical regime thus extending the standard
  setting of independent particles. This involves the large deviations
  rate function for the empirical measure which is in fact a richer
  observable than the hydrodynamic observables density and current.
  In the case in which the constraints are the initial and final
  density, we characterize the optimal measure for the Schr\"odinger
  problem.  We also introduce versions of the Schr\"odinger problem in
  which the constraints are related to the current and analyze the
  corresponding optimal measures.
\end{abstract}

\maketitle

\thispagestyle{empty}

\section{Introduction}

\subsection*{Schr\"odinger problem and large deviations}
According to the Schr\"odinger original formulation \cite{Sc}, his
problem amounts to the following.
``Imaginez que vous observez un syst\'eme de particules en diffusion, qui
soient en \'equilibre thermodynamique. Admettons qu'\`a un instant donn\'e
$t_0$ vous les ayez trouv\'ees en r\'epartition \'a peu pr\`es uniforme et
qu'\'a $t_1 > t_0$ vous ayez trouv\'e un \'ecart spontan\'e et consid\'erable
par rapport \`a cette uniformit\'e. On vous demande de quelle mani\`ere
cet \'ecart s'est produit. Quelle en est la mani\`ere la plus probable?''

A simple modelization of the ``syst\'eme de particules en diffusion''
is provided by a collection of $N$ independent Brownians to be
considered in the thermodynamic limit $N\to \infty$. The associated
natural observable is the empirical measure, defined by
\begin{equation}
  \label{em0}
  \ms R^N  = \frac 1N \sum_{i=1}^N \delta_{X^i(\cdot)}
\end{equation}  
where $X^i(t)$, $t\in[0,T]$, $i=1,\dots,N$ are the trajectories of the
particles in the time interval $[0,T]$. Note that $\ms R^N$ is a
random probability in the space of paths. We would then like to
predict the most probable behavior of $\ms R^N$ under the constraint
of the distributions of the particles at the initial and final times.
By large deviations theory, such behavior in the limit as $N\to
\infty$ is determined by the minimization of the large deviations rate
function of the family $(\ms R^N)_{N}$ with the given constraints.  By
using Sanov's theorem, which describes the large deviations of the
empirical measure for independent variables, the Schr\"odinger problem
is then formulated as a minimization of the relative entropy with
respect to the Wiener measure, see \eqref{rent}, under the constraint
of the initial and final distributions.  This procedure is referred to
as the Gibbs conditional principle.

The mathematical formulation of the Schr\"odinger problem is the
following. Fix $T>0$ and let $W$ be the distribution of a stationary
Brownian motion which we consider as the reference measure. Denote by
\begin{equation}
  \label{rent}
\Ent(R|W) = \int \!dW \, \frac {dR}{dW}\log   \frac {dR}{dW}\
\end{equation}
the relative entropy of the probability $R$ on the space of paths on
the time interval $[0,T]$ with respect to $W$.
For $t\in[0,T]$ we denote by $R_t$ the marginal at time $t$ of $R$,
that is $R_t$ is the probability on the configuration space that
describes the distribution at time $t$. 
Accordingly, the constraint on the initial and final distribution of
the particles is specified by the probabilities $\mu_0$ and $\mu_1$
on the configuration space. 
The solution to the the Schr\"odinger problem is then the optimal
probability for
\begin{equation}
  \label{ssip0}
  \inf\big\{ \Ent(R|W) ,\; R_0=\mu_0, R_T=\mu_1\big\}.
\end{equation}

We refer to \cite{Lrev} for a exhaustive review which illustrates the
rich mathematical structure of the Schr\"odinger problem.
In particular, this problem can be seen as an entropic regularization
of optimal transport and it is particularly apt to numerical
schemes. For these reasons, it has recently attracted much attention
within machine learning algorithms.

We here are mainly interested in relaxing the independence assumption
in the formulation \eqref{ssip0} and in considering versions of the
Schr\"odinger problem for interacting particles with different
constraints.
The case of a collection of Brownians with mean field interaction in
the It\^o-McKean regime has been considered in \cite{CL}, but the
present focus is the case of short range interactions in the
hydrodynamical scaling limit. This topic has been introduced in
\cite{CC} but there the variational problem is formulated in terms of
the hydrodynamic observables rather than the empirical measure.

\subsection*{Rate function for the empirical measure in the
  hydrodynamical regime} 

As in the case of independent particles, by large deviations theory,
the Schr\"odinger problem is naturally formulated as a minimization of
the rate function describing the fluctuations of the empirical
measure. Since particles are interacting, Sanov's theorem cannot be
applied we need to find the relevant cost functional.

When we consider stochastic interacting particle systems, the
hydrodynamical observables are the macroscopic density of particles
$\rho$ and the associated current $j$; by conservation of the number
of particles, the continuity equation $\partial_t \rho + \nabla \cdot
j =0$ holds.  The dynamic large deviations of these observables are
described by the functional
\begin{equation}
  \label{rfh0}
  I_\mathrm{dyn}(\rho,j) =\frac 14 \int_0^T\!dt\!\int\!dx\,
  \frac { | j+ D_\mathrm{h}(\rho) \nabla\rho |^2}{\sigma(\rho)}
\end{equation}
Here, the relevant transport coefficients are $\sigma$, the
\emph{mobility}, and $D_\mathrm{h}$, the \emph{hydrodynamic
diffusion}; in general they are symmetric matrices but we here
assume for simplicity that they are multiple of the identity. In terms
of the microscopic dynamics, the mobility is defined via a Green-Kubo
formula \cite{KL,Sp} and the diffusion satisfies the Einstein
relationship $D_\mathrm{h} = f'' \sigma$ where $f$ is the free energy
per unit of volume. In the case of independent particles we simply have
$\sigma(\rho)=\rho$ and $D_\mathrm{h}$ is constant.
The functional $I_\mathrm{dyn}$ is at the basis of the macroscopic
fluctuation theory \cite{BDGJLrev}.

As discussed by Quastel, Rezakhanlou, and Varadhan for the exclusion
process \cite{QRV}, the dynamical contribution to the rate function
for the empirical measure, as defined in \eqref{em0}, in the
hydrodynamical scaling limit is given by
\begin{equation}
  \label{qrv0}
  \mf F_\mathrm{dyn}(R)
  = I_\mathrm{dyn}\big( \rho^R, j^R \big)
  + \Ent\big( R \, \big|\, P(R)\,\big). 
\end{equation} 
Here $(\rho^R,j^R)$ is the pair density/current corresponding to the
measure $R$, in particular $\rho^R_t$ is the density of $R_t$, and
$P(R)$ is a Markovian measure on the space of paths with single time
marginals and expected current given by $(\rho^R,j^R)$. We refer to
the discussion in Section~\ref{s:rf} for the precise definition which
involves the \emph{self-diffusion} $D_\mathrm{s}$ that encodes the
asymptotic variance of a tagged particle.  The first term on the right
hand side of \eqref{qrv0} describes the fluctuations of the
hydrodynamical observables, that could be thought as Eulerian
observables, while the second term takes into account the actual path
of single particles, so that it could be thought as reflecting a
Lagrangian viewpoint.  As follows directly from \eqref{qrv0}, when $R$
varies within a suitable class of Markovian measures then the value
$\mf F_\mathrm{dyn}(R)$ can be obtained from the hydrodynamical rate
function in \eqref{rfh0}. On the other hand, with probability
exponentially small in the total number of particles it is also
possible to realize deviations of the empirical measure that are not
Markovian.  In this respect, the functional in \eqref{qrv0} is the
proper extension to interacting particles in the hydrodynamical
scaling limit of $\Ent(\cdot|W)$; indeed, as we here show, for
independent particles it reduces to that functional.

\subsection*{Standard Schr\"odinger problem for interacting particles}

The Schr\"odinger problem for stochastic lattice gases has been
originally introduced by Chiarini, Conforti, and Tamanini \cite{CC} as
a minimization of the hydrodynamical rate function \eqref{rfh0} with
the constraint on the initial and final densities. The corresponding
minimizer solves the canonical equations corresponding to the
Hamiltonian structure obtained by regarding $I_\mathrm{dyn}$ as an
action functional \cite{BDGJLrev}.

The main novel point of the present
work is the formulation of this problem as a minimization problem on
the path measures rather than hydrodynamical observables, that is as
\begin{equation}
  \label{ss0}
  \inf\big\{
  \mf F_\mathrm{dyn} (R) ,\; R_0=\mu_0, \, R_T=\mu_1\big\}.
\end{equation}  
Within this setting, such optimization actually provides an extension
of the classical Schr\"odinger problem \eqref{ssip0}.
We here characterize the optimal measure for \eqref{ss0} as a
time-inhomogeneous diffusion which can be realized as the typical
behavior of the empirical measure for the microscopic dynamics
perturbed by a suitable gradient external field.

A most interesting feature of the Schr\"odinger problem for
independent particles \eqref{ssip0} is that its solution can be
expressed in terms of a pair of dual potentials, see
e.g.\ \cite{Lrev}.  In same respects, as discussed in \cite{CC}, this
feature still holds for interacting particles, we here provide a
pathwise interpretation in terms of the time reversal of the optimal
measure for \eqref{ss0}.

\subsection*{Versions of the Schr\"odinger problem with constraints on
  the current}

According to the physical interpretation of interacting particle
systems as basic models for out of equilibrium processes, like mass
transportation, it appears natural both from a conceptual
point of view and in the description of an actual experiment to
consider versions of the Schr\"odinger problem with constraints on the
current. In the context of the hydrodynamical function \eqref{rfh0},
this issue has become a central topic in non-equilibrium statistical
mechanics \cite{BDGJLrev}. 

In the present work, we consider instead such optimization at the level
of path measures, the corresponding relevant functional being \eqref{qrv0}.
Referring to Section~\ref{s:spcec} for
other constraints, we here briefly discuss the case in which we
introduce a constraint on the time averaged expected current.
We thus consider the optimization
\begin{equation}
  \label{spec0}
  \frac 1T \inf\Big\{
  \mf F_\mathrm{dyn} (R) , \ 
  \frac 1T \int_0^T\!dt\, j^R_t = \bar J \Big\}
\end{equation}  
where the datum $\bar J$ is a given divergence free vector field and
we recall that $j^R$ denotes the expected current under $R$.

We here characterize the optimal measure for this problem in terms of a
variational problem on the hydrodynamical functional \eqref{rfh0}
and discuss its asymptotic behavior in the limit $T\to \infty$.  In
the case of interacting particles, in this asymptotics dynamical phase
transitions may occur, in the sense that the minimizer for
\eqref{spec0} exhibits a non-trivial time dependence. 
In particular, these phase transitions imply the breaking of
time translation invariance \cite{BDGJLptcur}.

\subsection*{Outline}
We next briefly outline the organization of the present work, poiting
out the parts that  review the existing literature and the original
ones. 

In Section~\ref{s:rf} we review the definition of the hydrodynamical
observables, empirical density and current, and the corresponding law
of large numbers, the hydrodynamical scaling limit.
We also introduce the main observable used in the formulation of the
Schr\"odinger problem, that  is the empirical measure, and characterize
its limiting behavior.
We refer to \cite{BDGJLrev,KL,Sp} for the
hydrodynamical observables and to \cite{R,Vrev} for the  empirical measure
in the context of hydrodynamical scaling limits. 

In Section~\ref{s:ld} we review the large deviation principle for the
observables introduced in Section~\ref{s:rf}, we refer to
\cite{BDGJLrev,KL,Sp,QRV} for a more detailed discussion. In this
Section we also provide, as an original result, a novel
characterization of the rate function obtained in \cite{QRV} for the
so-called zero-range process. Finally we briefly discuss the large
deviations of the empirical measure in the context of mean field
models and, again as a original result, we connect the form of the
rate function in \cite{QRV} with the one derived in \cite{CL,Fi}.

Section~\ref{s:spip} is the main content of the present work. We there
review the formulation of the Schr\"odinger problem for interacting
particles at the level of the hydrodynamical observables introduced in
\cite{CC}. We then introduce its novel formulation at the
level of the empirical measure; in particular the standard
formulation of the  Schr\"odinger problem, see e.g.\ \cite{Lrev}, is
obtained in the particular case of non-interacting particles.
As a simple but original result, we show that the optimal measure for
the Schr\"odinger problem can be characterized in terms of the
canonical equations describing the large deviations of the 
hydrodynamical observables \cite{BDGJLrev}.

In Section~\ref{s:spcec} we introduce versions of Schr\"odinger problem
with constraints on the current rather than the initial and final
densities and characterize the corresponding optimal measures.
While closely related with the main problems in the Macroscopic
Fluctuations Theory \cite{BDGJLrev,Der}, formulations of the
Schr\"odinger problem with constraints on the current appears to have
not been considered before in the literature.

In Appendix~\ref{s:rap}, written by R.\ Chetrite, it is shown that the
representation of the rate function in \cite{QRV} as a
\emph{non-linear} relative entropy holds in the general context of
hydrodynamical scaling limit and not only for the zero-range process
discussed in Section~\ref{s:ld}.  In the general case the reference
measures depend however not only on the density, as in the zero-range
case, but also on the current.

\bigskip
Our aim is here to present the main conceptual points in formulating
the Schr\"o\-din\-ger problem for interacting particle and we are
going to be somewhat cavalier on the technical details.

\section{Empirical observables, hydrodynamical limit and tracer dynamics} 
\label{s:rf}

In this section we introduce the basic empirical objects associated to
a particle system, describe the
hydrodynamic limit, illustrate the behavior of tagged particles and recall a related
law of large numbers in \cite{QRV}.

\subsection*{Microscopic models and empirical observables}
As basic microscopic model we consider a system of interacting
particles evolving according to some stochastic Markovian dynamics,
which we assume to be reversible with respect to a Gibbs measure.
More specifically, these systems are either stochastic
lattice gases, in which particles perform interacting random walks on
the lattice, or interacting Brownians.  We assume that the particles
move on the $d$-dimensional torus $\bb T^d_\ell$ of side
$\ell$, that is $\bb T^d_\ell := \big( \bb R / (\ell\bb Z)\big)^d$,
$\ell\in\bb N$.  The corresponding paths of the $N$ particles are
denoted by $(X^i(t))_{t\ge 0}$, $i=1,\ldots,N$.  In the case of
lattice gases the path is piecewise constant and takes values in
$\bb T^d_\ell \cap \bb Z^d$ while it is a continuous trajectory in the
case of interacting Brownians.  In the context of hydrodynamic limits,
we are interested in the case of a fixed density, i.e.\ in the limit
in which $N,\ell$ diverge and $N\ell^{-d}$ converges to a
non-vanishing density $m\in (0,\infty)$.

\subsubsection*{Empirical density}
Given a macroscopic time interval $[0,T]$, the \emph{empirical density}
is defined as the density of particles in a `small'
macroscopic volume. More precisely, we set
\begin{equation}
  \label{ed}
  \pi^\ell_t := \frac 1{\ell^d} \sum_{i=1}^N
  \delta_{\frac 1{\ell} X^i({\ell^2\,t})},
    \qquad t\in[0,T]  
 \end{equation}
where we have performed the diffusive rescaling of space and time. 
We understand that $(\pi^\ell_t)_{t\in[0,T]}$ is a random element in the
Skorokhod space $D([0,T];\ms M_+ (\bb T^d))$ where hereafter
$\ms M_+(\mc X)$ denotes the set of positive finite Borel measures on the
Polish space $\mc X$  endowed with the topology induced by weak convergence.
As the particles labels are not relevant for \eqref{ed}, in the case
of stochastic lattice gases the empirical density can be deduced from
the corresponding occupation numbers.
Denoting by $\eta_x(t)$ the number of particles at the site $x$ at
time $t\ge 0$, we in fact have
\begin{equation*}
  \pi^\ell_t := \frac 1{\ell^d}
  \sum_{x\in\bb T^d_\ell\cap \bb Z^d} \eta_x({\ell^2t})
  \delta_{\frac{x}{\ell}},     \qquad t\in[0,T].  
\end{equation*}

\subsubsection*{Empirical current}

The \emph{empirical current} is the observable of the microscopic
dynamics that accounts for the displacement of the particles. It is a
random element $\ms J^\ell$ in a suitable space of distributions $\mf J$ dual
to the smooth functions $\omega\colon [0,T]\times \bb T^d\to \bb R^d$.
In the lattice gases case it is defined by taking the spatial average
of the net flow of particle across the edges of
$\bb T^d_\ell\cap \bb Z^d$, we refer e.g.\ to
\cite{BDGJLrev,Sp} for the precise definition.
For interacting Brownians it is the empirical average
\begin{equation}
  \label{ec1}
  \ms J^\ell = \frac 1{\ell^d} \sum_{i=1}^N \mc J_i^\ell
\end{equation}
where $(\mc J_i^\ell)_{i=1,\ldots,N}$ are the \emph{stochastic currents},
in the sense of \cite{FGGT}, associated to the particles, i.e.\ for
smooth $\omega\colon [0,T]\times \bb T^d\to \bb R^d$
\begin{equation}
  \label{ec2}
  \mc J_i^\ell (\omega) = \frac 1\ell
  \int_0^{T}\! \omega_t(\ell^{-1}X^i(\ell^2 t))\circ dX^i(\ell^2 t),
  \qquad i=1,\ldots,N,
\end{equation}
where $\circ$ denotes the Stratonovich integral and we have performed
the diffusive rescaling of space and time. 
In both cases, in view of the chain rule for the Stratonovich
integral, the continuity equation
\begin{equation}
  \label{ce}
  \partial_t \pi^\ell +\nabla\cdot \ms J^\ell =0
\end{equation}
holds weakly with probability one with respect to the microscopic dynamics.

\subsubsection*{Empirical measure}
The \emph{empirical measure} is a finer observable of the microscopic
dynamics which retains the information on time correlations also
in the limit of infinitely many particles. It is defined by
\begin{equation}
  \label{em}
  \ms R^\ell := \frac 1{\ell^d} \sum_{i=1}^N
  \delta_{\frac 1{\ell} X^i(\ell^2 \cdot)}
\end{equation}
that, given $T>0$, we regard as a random element in
$\ms M_+ \big( D([0,T];\bb T^d)\big)$.  While the empirical density
$\pi^\ell$ is random path on the space of measures, the empirical
measure $\ms R^\ell$ is a random measure on the space of paths;
accordingly, the $\delta$ measures in \eqref{ed} are measures on
$\bb T^d$ (points) while those in \eqref{em} are measures on
$D([0,T];\bb T^d)$ (paths).

The empirical density $(\pi^\ell_t)_{t\in[0,T]}$ can be recovered from
the empirical measure by considering the family of the single time
marginals of $\ms R^\ell$.  In order to recover also the empirical
current, for a path $X\colon [0,T]\to \bb T^d$, let $\mc J(X)\in \mf J$ the
corresponding stochastic current. For a continuous path, as in the
case of interacting Brownians, $\mc J(X)$ is defined, according to
\cite{FGGT}, by requiring that for smooth $\omega\colon
[0,T]\times\bb T^d\to \bb R^d$ 
\begin{equation}
  \label{stcur}
  \big(\mc J(X)\big) (\omega) = \int_0^T\!\omega_t(X_t)\circ dX(t).
\end{equation}
For piecewise constant paths, as in the case of stochastic lattice
gases,  $\mc J(X)$ is instead defined, as in \cite{BGL}, by looking at the
net flow of particles. 
In either case, as follows directly from \eqref{ec1}--\eqref{ec2} and \eqref{em}
\begin{equation}
  \label{scfem}
  \ms J^\ell = \int \! \ms R^\ell(dX) \, \mc J (X).
\end{equation}

To illustrate the difference between the empirical density and measure,
it is instructive to compare two models, the standard simple
exclusion, in which the interaction is specified by a hard core
constraint, and the one in which the particles exchange their
position, that can be obtained from the so-called stirring process \cite{Gri}.
The hydrodynamical observables, empirical density and current, are the
same for both models while the empirical measure is clearly different.

\subsubsection*{Empirical stochastic current}
Finally we can define another empirical observable, that we refer to
as the \emph{empirical stochastic current} and denote by
$\hat{\ms J}^\ell$.
It is the random element of $\ms M_+(\mf J)$ defined by
\begin{equation}\label{def-tagged}
  \hat{\ms J}^\ell= \ms R^\ell\circ \mc J^{-1}\,,
\end{equation}
with $\mc J =\mc J(X)$ as defined in \eqref{stcur}.
Denoting by $J$ a generic element of $\mf J$,
we have $\ms J^\ell=\int\! \hat{\ms J}^\ell(dJ)\, J$, that is the
expectation of the empirical stochastic current is the empirical current.

We remark that, as before, the empirical current is the same when the
exclusion process is defined either by hard core or by stirring, while
instead the empirical stochastic current is different in the two cases.

\subsection*{Hydrodynamical limit and tracer dynamics}

The hydrodynamical limit a\-moun\-ts to the law of large numbers for the
empirical observables introduced before.  We assume that at the
initial time they converge to a deterministic limit and want to describe
the corresponding limiting evolution.

\subsubsection*{Hydrodynamical limit for the empirical density}
This is the standard formulation of the hydrodynamical limit.
Referring to the monographs \cite{KL,Sp} for a detailed exposition,
under suitable assumptions on the initial distribution of particles,
the empirical density $(\pi^\ell_t)_{t\in[0,T]}$
converges in probability to the absolutely continuous measure
$(\rho_t\,dx)_{t\in[0,T]}\in C\big([0,T];\ms M_{+} (\bb T^d)\big)$
where $dx$ denotes the uniform measure on $\bb T^d$ and
$(\rho_t)_{t\in[0,T]}$ solves the non-linear diffusion equation
\begin{equation}
  \label{nde}
  \partial_t \rho= \nabla\cdot \big( D_\mathrm{h}(\rho)\nabla \rho\big).
\end{equation}
The \emph{hydrodynamical diffusion coefficient} $D_\mathrm{h}$
can be obtained from the microscopic dynamics by a Green-Kubo
formula \cite{KL,Sp}.
While $ D_\mathrm{h}$ is in general a positive 
symmetric $d\times d$ matrix, for simplicity we here assume that it is
a multiple of the identity.

\subsubsection*{Hydrodynamical limit for the empirical density and
  current} 
As discussed in \cite{BDGJLldcur} for the case of the simple exclusion
process, see also \cite{BGL} for a more accurate discussion on the
tightness of the empirical current, the hydrodynamical limit can be
also formulated in term of the empirical density and current.
More precisely, as $\ell\to \infty$ the pair
$\big((\pi^\ell_t)_{t\in[0,T]}, \ms J^\ell\big)$ converges in
probability to 
$\big((\pi_t)_{t\in[0,T]}, \ms J\big)$ where, as before,
$\pi_t=\rho_t\,dx$, $t\in [0,T]$ and $\ms J$ can be represented by a
$L^2$ vector field, i.e.\ there exists $j\in L^2\big([0,T]\times \bb
T^d;\bb R^d)$ such that for any smooth $\omega \colon [0,T]\times \bb T^d\to
\bb R^d$ 
\begin{equation}
  \label{rJl}
  \ms J(\omega) = \int_0^T\!dt \!\int\!dx\, j_t(x) \cdot \omega_t(x)
\end{equation}
in which $\cdot$ denotes the inner product in $\bb R^d$.
Finally, the pair  $(\rho_t,j_t)_{t\in[0,T]}$ satisfies
\begin{equation}
  \label{hlrj}
  \begin{cases}
    \partial_t \rho +\nabla\cdot j =0, \\
    j= -D_\mathrm{h}(\rho)\nabla\rho.
  \end{cases}
\end{equation}
In which we emphasize that the continuity equation holds exactly also
when $\ell$ is finite while the Fick's law
$j=-D_\mathrm{h}(\rho)\nabla\rho$ is achieved only in the limit
$\ell\to\infty$.

\subsubsection*{Equilibrium tracer dynamics}

In the standard formulation, the tracer dynamics is specified by
selecting a distinguished particle, the so-called \emph{tagged
  particle}, and looking for its statistics in the limit of infinitely
many particles.  When the other particles are sampled according to the
equilibrium (homogeneous) distribution with density $m\in(0,\infty)$,
by applying the general framework in \cite{KV}, it has been shown that
- under diffusive rescaling - the position of the tagged particles
converges to a Brownian motion with diffusion coefficient
$D_\mathrm{s}(m)$ that can be identified from the microscopic dynamics
by a suitable variational formula.  We refer to the monographs
\cite{KLO,Sp} for a detailed exposition.
In the case of the simple exclusion process, $D_\mathrm{s}$ vanishes
in the one-dimensional case for nearest neighbors jumps. Indeed, in
this case the variance of the tagged particles is of the order
$\sqrt{t}$ \cite{Ar}.
On the other hand both for $d\ge 2$ and for $d=1$ with next nearest
neighbors jumps, $D_\mathrm{s}>0$ \cite{Sp}.    
In general, the \emph{self-diffusion} coefficient $D_\mathrm{s}$
cannot be expressed in a closed form, a notable exception, as next
discussed, is the zero-range process.  We finally remark that for the
stirring process we simply have $D_\mathrm{h}=D_\mathrm{s}=1$.

\subsubsection*{Law of large numbers for the empirical measure}
In the context of hydrodynamical scaling limits, a detailed proof of
this result has been carried out in the case of the exclusion process
\cite{Q1,R}. The corresponding statement next outlined should however
be of general validity.

We denote by $\ms R^\ell_0$ the marginal distribution at the initial
time of the empirical measure that in fact coincides with
$\pi^\ell_0$. We assume the initial distribution of the particles is
chosen so that $\ms R^\ell_0$ converges in probability to a non-random
measure $\bar\rho\, dx$ on $\bb T^d$ which is absolutely continuous
with respect to $dx$ and denote by $m \in(0,\infty)$ its total
mass.  The law of large numbers for the empirical measure $\ms R^\ell$
is the non-random measure $P^{\bar\rho}$ on $C([0,T];\bb T^d)$ of mass
$m$ that is next described as a the law of a \emph{non-linear} Markov
diffusion process.

For a Polish space $\mc X$, we denote by $\ms M_m(\mc X)\subset\ms
M_+(\mc X)$  the set of positive measures with mass $m$.
Given $P \in \ms M_+ \big( C([0,T];\bb T^d) \big)$ we further denote
by $P_{t} \in \ms M_+ (\bb T^d)$, $t\in [0,T]$, the family of its
single time marginals.  Then the measure $P^{\bar\rho}\in \ms M_m \big(
C([0,T];\bb T^d) \big)$ is the law of the  time-inhomogeneous
diffusion with time dependent generator \cite{QV} 
\begin{equation}
  \label{tdg}
    L_t = \nabla \cdot D_\mathrm{s} (\rho_t) \nabla 
    + \big[ D_\mathrm{s} (\rho_t) -  D_\mathrm{h} (\rho_t)\big]
    \frac {\nabla \rho_t}{\rho_t} \cdot \nabla 
\end{equation}
and initial condition $(P^{\bar\rho})_0 =\bar\rho \, dx$. Here, the family of 
densities  $(\rho_t)_{t\in [0,T]}$ is determined by the compatibility
condition $(P^{\bar\rho})_t = \rho_t\, dx$. In particular,
$P^{\bar\rho}$ is characterized as the solution to a fixed point problem.

By a direct computation, the density of the single time
marginal of $P^{\bar\rho}$ evolve according to the non-linear
diffusion equation \eqref{nde}. 
This implies that, alternatively to the fixed point characterization,
the measure $P^{\bar\rho}$ can be described as the law of the
time-inhomogeneous diffusion with initial distribution $\bar\rho\,dx$
and time-dependent generator \eqref{tdg} where $(\rho_t)_{t\in[0,T]}$
is the solution to the hydrodynamic equation \eqref{nde} with initial
condition $\bar\rho$.

\smallskip
As an argument to justify formula \eqref{tdg} we show that, for fixed
$\bar\rho$, \eqref{tdg} is selected among diffusion processes by
imposing three natural conditions.  Consider a general time dependent
diffusion on $\mathbb T^d$ having diffusion matrix $D=D(x,t)$ and
drift $b=b(x,t)$. Its time dependent generator has the form
\begin{equation}\label{gen-diff}
  \mathcal L f= \tr \big( D \, \partial^2 f \big)
  +b\cdot \nabla f
\end{equation}
where $\partial^2 f$ denotes the Hessian of $f$. The three conditions
are the following; the first two  are indeed strict constraints
related to the definitions of the empirical measure, density and
current.  
\begin{itemize}
\item For each $t\in[0,T]$ the density of $(P^{\bar\rho})_t$ coincides
  with the solution at time $t$ to \eqref{nde} with initial condition
  $\bar\rho$.
\item The expected current $\int \! P^{\bar\rho}(dX) \, \mc J (X)$
  coincides with the hydrodynamic current and is therefore represented
  by the $L^2$ vector field $j=-D_{\mathrm{h}}(\rho)\nabla \rho$,
  where $\rho$ is the solution of \eqref{nde} with initial condition
  $\bar\rho$.
\item The value $D(x,t)$ of the diffusion coefficient at the
  space-time point $(x,t)$ in \eqref{gen-diff} is given by
  $D_{\mathrm{s}}(\rho_t(x))$ where $\rho$ is the solution of
  \eqref{nde} with initial condition $\bar\rho$.
\end{itemize}
Recall that a diffusion process having generator \eqref{gen-diff} and
with density of the distribution at time $t$ given by $\alpha_t$, has
an expected current with a representative given by $-\nabla
\cdot \big(D\,\alpha\big) +\alpha \,b$ (see \cite{FGGT}); where
$\nabla \cdot \big(D\alpha\big)$ has components
$\sum_j\partial_{x_j}\left(D_{i,j}\alpha\right)$, $i=1,\dots,d$.  The
first condition implies therefore that the density of the expected
current is
\begin{equation}\label{intju}
  -\nabla \cdot \big(D\,\rho\big) +\rho \,b
\end{equation}
The third condition imposes that in \eqref{intju} we have
$D=D_{\mathrm{s}}(\rho)$ and finally the second condition gives the
relationship
\begin{equation*}
-\nabla \cdot \big(D_{\mathrm{s}}(\rho)\,\rho\big) +\rho
\,b=-D_{\mathrm{h}}(\rho)\nabla \rho\,, 
\end{equation*} 
that identifies uniquely the vector $b$ and gives as a result the
diffusion with generator \eqref{tdg}.

\subsubsection*{Law of large numbers for the empirical stochastic
  current}
In view of \eqref{def-tagged} and the law of large numbers for the
empirical measure, we deduce that as $\ell \to \infty$ the empirical
stochastic current $\hat {\ms J}^\ell$ converges to $P^{\bar\rho}\circ
\mc J^{-1}$, with $\mc J$ as defined in \eqref{stcur}. 
That is, the limit of empirical stochastic current is the stochastic
current associated to the non-linear Markov diffusion
$P^{\bar\rho}$. On the technical side we mention that, as the
map $X\mapsto \mc J(X)$ is not continuous, the previous statement is
not a direct consequence of the law of large numbers for the empirical
measure. We do not however discuss here the details.


\section{Large deviations}
\label{s:ld}

The main aim of this section is to introduce the functional derived in
\cite{QRV} as the large deviation rate function for the empirical
measure associated to the exclusion process in the hydrodynamical
scaling limit. The versions of the Schr\"odinger problem that we will
next consider will be in fact formulated as minimization problems of
this functional.
We do not even attempt to sketch the arguments in \cite{QRV}, that
rely on the heavy machinery of hydrodynamical limits for 
non-gradient models and on an projective limit procedure, but
we do not limit the discussion to the case of the exclusion process
and consider general interacting particle systems.  With respect to
the original formulation in \cite{QRV}, following the general ideology
of \cite{BDGJLrev}, we emphasize the role played by the current.
We here also show that the recipe in \cite{QRV} can be applied beyond
the context of hydrodynamical limits, in particular to the case of
mean field interacting Brownians for which it outputs the functional
derived in \cite{CL}.
Finally, in the case of the zero range process, we obtain a simpler
expression for the functional in \cite{QRV}, which has the same flavor
as the one derived for mean field interacting Brownians.

\subsection*{Large deviations for the empirical density and current} 
Since the original derivation in \cite{KOV}, the large deviation
principle for the empirical density has been achieved for several
models, we refer to the bibliographic remarks in \cite{KL}.
As discussed in \cite{BDGJLrev}, the corresponding rate functional can
be formulated by considering the joint large deviations of the
empirical density and current. Referring to \cite{BDGJLldcur,BGL} for
the actual derivation of this functional in the context of the
exclusion process, we next present its general form.

To describe the rate function, we need another transport coefficient,
the so-called \emph{mobility}, which encodes the linear response of
the interacting particle system to an applied external field. If we
denote by $E=E_t(x)$ the applied field and by $\rho=\rho_t(x)$ the
local density then the induced local drift is $\sigma(\rho) E$, with
$\sigma$ being the mobility coefficient. As for the hydrodynamical
diffusion coefficient, $\sigma$ is a $d\times d$ symmetric positive
matrix but we here assume that is a multiple of the identity.
The local reversibility of the microscopic dynamics imply that
$D_\mathrm{h}$ and $\sigma$ are related by the \emph{Einstein
  relationship}
\begin{equation}
  \label{er}
  D_\mathrm{h}(\rho) = f''(\rho) \sigma(\rho), \qquad \rho\in
  [0,\infty), 
\end{equation}
where $f$ denotes the free energy per unit of volume associated to the
underlying Gibbs measure. From a mathematical viewpoint, as discussed
in \cite{Sp}, the Green-Kubo formula mentioned above really provides a
variational formula for the mobility $\sigma$ and \eqref{er} is then
the appropriate definition of the hydrodynamical diffusion coefficient. 

Recall that $\mf J$ is the space of distributions in which the empirical
current takes its value. The \emph{hydrodynamical rate function} is the map
$I\colon D\big([0,T]; \ms M_+(\bb T^d)\big)\times \mf J \to
[0,+\infty]$ defined as follows.  It is infinite unless
$\pi \in C\big([0,T]; \ms M_+(\bb T^d)\big)$ with
$\pi_t =\rho_t \,dx$, $t\in[0,T]$, and $\ms J \in \mf J$ can be
represented by a $L^2$ vector field $j$ as in \eqref{rJl} which
satisfy weakly the continuity equation
$\partial_t \rho +\nabla\cdot j =0$.
In this case $I(\pi,\ms J)$ is given by
\begin{equation}
  \label{rfh}
  I(\pi,\ms J) =I_\mathrm{in}(\rho_0) + I_\mathrm{dyn}(\rho,j)
\end{equation}
where $I_\mathrm{in}$ takes into account the initial distribution of
the particles and 
\begin{equation}
  \label{rfhdyn}
  I_\mathrm{dyn}(\rho,j) =\frac 14 \int_0^T\!dt\!\int\!dx\,
  \frac { | j+ D_\mathrm{h}(\rho) \nabla\rho |^2}{\sigma(\rho)}.
\end{equation}

In general, $I_\mathrm{in}$ is the large deviation rate function for
the family $\big(\pi^\ell_0\big)_{\ell>0}$ when the particles are
sampled according to the specified initial datum; we discuss its
explicit form only for two special choices of the initial distribution
of the particles.  If we consider a \emph{deterministic} choice such
that the corresponding empirical density $\pi^\ell_0$ converges to
$\bar\rho\, dx$ then $I_\mathrm{in}(\rho_0)$ is infinite unless
$\rho_0=\bar\rho$ and zero in this case.  If we distribute the
particles according to the Gibbs measure with density $m$ then
\begin{equation}
  \label{Iincan}
  I_\mathrm{in}(\rho_0)= \int\!dx \big[ f(\rho_0)-f(m) -f'(m)
  (\rho_0-m)\big] 
\end{equation}
where we recall that $f$ denotes the free energy per unit of volume. 

The heuristic argument leading to \eqref{rfhdyn} is the
following. Consider the microscopic dynamics perturbed by an external
field $2E=2E_t(x)$. Its typical behavior is then described by the
perturbed hydrodynamical equation
\begin{equation}
  \label{phe}
  \begin{cases}
    \partial_t \rho +\nabla\cdot j =0, \\
    j= -D_\mathrm{h}(\rho) \nabla\rho +2\sigma(\rho) E.
  \end{cases}
\end{equation}
We next choose $E$ so that $(\rho,j)$ is the prescribed path and
compute the work done by the external field,
\begin{equation}\label{phe-I}
  I_\mathrm{dyn}(\rho,j) = \int_0^T\!dt\int\!dx\,
  \sigma(\rho) |E|^2.
\end{equation}

\subsection*{Large deviations for the empirical measure} 
In the context of the hydrodynamical limit this issue has been
originally discussed in the case of the simple exclusion process by
Quastel, Rezakhanlou and Varadhan \cite{QRV}.  The proof is carried out
when $d\ge 3$ as it is needed a regularity property of the
self-diffusion $D_\mathrm{s}$ that it has been proven only in that
case. The case of one-dimensional Brownians interacting by partial
reflection is discussed in \cite{Seo}. By using the results in
\cite{GJL}, the large deviations principle can be obtained also for
the zero range process. As discussed in the review \cite{Vrev}, the
rate function in \cite{QRV}, that we next introduce, should be of
general validity.

Recall that, as defined in \eqref{stcur}, $\mc J(X)$ is the stochastic
current associated to the path $X\in C\big([0,T];\bb
T^d)$. Furthermore, $(R_t)_{t\in[0,T]}\in C\big([0,T];\mathcal M_+(\bb
T^d)\big)$ is the family of the single time marginals of
$R\in \ms M_+\big(C\big([0,T];\bb T^d)\big)$.  Letting
$\ms J^R :=\int \! R(dX) \, \mc J(X)$ be the expectation of the
stochastic current, then the continuity equation
$\partial_t R_t +\nabla\cdot \ms J^R=0$ holds weakly.  Assume that
$R_t =\rho_t \,dx$, $t\in [0,T]$, and $\ms J^R$ can be represented by
the $L^2$ vector field $j^R$; we then let $E^R$ be the external field
such that $j^R = -D_\emph{h}(\rho) \nabla\rho + 2 \sigma(\rho) E^R$.
We then denote by $P(R)\in \ms M_+\big(C\big([0,T];\bb T^d)\big)$ the
law of the time-inhomogeneous diffusion with time dependent generator
\begin{equation}
  \label{tdgp}
    L^R_t = \nabla\cdot D_\mathrm{s} (\rho_t) \nabla 
    + \Big\{
    \big[ D_\mathrm{s} (\rho_t) -  D_\mathrm{h} (\rho_t)\big]
    \frac {\nabla \rho_t}{\rho_t} 
    + 2 \frac {\sigma(\rho_t)}{\rho_t} E^R_t \Big\}\cdot \nabla
\end{equation}
and initial condition $\big(P(R)\big)_0 =R_0$.
Observe that, as follows by a direct computations, $R$ and $P(R)$ have
the same family of single time marginals which in fact evolve
according to the perturbed hydrodynamical equation \eqref{phe}.
On the other hand, while $P(R)$ is Markovian $R$ is not required to be
Markovian.  The measure $P(R)$ can thus be regarded as a Markovian
measure with the same single time marginals and same expected current as
$R$ and, following the same arguments after \eqref{tdg}, can be characterized as
the unique diffusion process \eqref{gen-diff} with these two features and having
a time dependent diffusion coefficient such that $D(t)=D_{\mathrm{s}}(\rho_t)$.  
Finally, if $E^R=0$ then $P(R)=P^{\rho_0}$ is the measure
obtained as the law of large numbers for the empirical measure.

The \emph{QRV functional} is the map
$\mf F \colon \ms M_+\big(D\big([0,T];\bb T^d)\big) \to [0,+\infty]$
defined as follows. It is infinite unless
$R \in \ms M_+\big(C\big([0,T];\bb T^d)\big)$ and in this case it is
given by
\begin{equation}
  \label{rfqrv}
  \mf F(R) = I\big( (R_t)_{t\in[0,T]}, \ms J^R \big)
  + \Ent\big( R \, \big|\, P(R)\,\big) 
\end{equation}
where $I$ is the rate function introduced in \eqref{rfh} and
$\Ent\big(\mu|\nu) =\int\! d\mu \log \frac{d\mu}{d\nu}$ is the
relative entropy between the positive measures $\mu$ and $\nu$ which
have the same mass.
Note that, as follows from the definition of the hydrodynamical rate
function $I$, the measure $P(R)$ is well defined when
$I((R_t)_{t\in[0,T]}, \ms J_R \big) <+\infty$.  As proven in
\cite{QRV} in the case of the exclusion process, $\mf F$ is the large
deviation rate function of the family $\big(\ms R^\ell\big)_{\ell>0}$.

We remark that if the measure $R$ is Markovian then $\mf F(R)$ is
finite only if $R$ is absolutely continuous with respect to
$P(R)$ and, if this is the case, then $R$ is actually equal to $P(R)$
so that $\mf F(R) = I\big( (R_t)_{t\in[0,T]}, \ms J^R
\big)$. Therefore, the second term on the right hand side of
\eqref{rfqrv} is relevant only to describe the occurrence of
non-Markovian measures.

As readily follows from \eqref{rfqrv}, the hydrodynamic rate function
$I$ can be recovered from the QRV rate function $\mf F$ by
projection. 
On the technical side we mention however that, as the map
$R\mapsto \ms J^R$ is not continuous, the corresponding large
deviations principle cannot really obtained by contraction and the
construction of suitable exponentially good approximation of the
empirical current is needed.

\subsection*{The QRV functional for the zero range process}

The zero range process, see e.g.\ \cite{KL}, is a stochastic lattice
gas in which the rate at which a particles jumps from the site $x$ to
a neighboring site $y$ depends only on the number of particles at
$x$; more precisely if there are $k$ particles at site $x$ the jump
rate is $g(k)$, $k\in\bb Z_+$. It is reversible with respect to the
one-family product probability measures $(\nu_\varphi)_{\varphi\ge 0}$
on $(\bb Z_+)^{\bb T^d_\ell \cap \bb Z^d}$ with marginals
\begin{equation*}
  \nu_\varphi (\eta_x=k) =\frac 1{Z(\varphi)} \frac {\varphi^k}{g(k)!},
\end{equation*}
where we recall that $\eta_x$ are the occupations numbers, $Z(\varphi)$
is the appropriate normalization constant, and
$g(k)!= g(k)g(k-1) \cdots g(1)$.

We denote by $\phi\colon \bb R_+\to \bb R_+$ the function such that
$\sum_k k \, \nu_{\phi(\rho)} (\eta_x=k)=\rho$.
As discussed in \cite{KL}, to which we also refer for the required
conditions on the jump rates $g$,
the hydrodynamic transport coefficient are given by
$D_\mathrm{h}(\rho) = \phi'(\rho)$ and $\sigma(\rho)=\phi(\rho)$ while
the free energy per unit of volume $f$ satisfies
$f''(\rho) =\phi'(\rho)/ \phi(\rho)$. Observe that in the case $g(k)=k$,
$k\in \bb Z_+$, we have non-interacting symmetric random walks;
accordingly, $D_\mathrm{h}(\rho)=1$, $\sigma(\rho)=\rho$, and
$f(\rho)=\rho\log\rho$.
Another special feature of the zero range process is that the position
of the tagged particle is a martingale. By direct computations, see
e.g.\ \cite{JLS}, is then possible to express also the self-diffusion
coefficient in a closed form, $D_\mathrm{s}(\rho) = \phi(\rho)/\rho$
where we understand $D_\mathrm{s}(0)=1$.

In view of the peculiarities above described, for the zero range
process the QRV functional admits a simpler formula, somewhat
analogous to the rate function derived in \cite{CL} for mean field
interacting Brownians, in which we get rid both of the hydrodynamical
rate function and of the external field.  Given the family of
densities $\rho=(\rho_t)_{t\in[0,T]}$, we denote by
$Q(\rho)\in \ms M_+\big(C\big([0,T];\bb T^d)\big)$ the law of the
time-inhomogeneous diffusion on $\bb T^d$ with time dependent
generator
\begin{equation}
  \label{tdgzr}
  L^\rho_t = \frac{\phi(\rho_t)}{\rho_t} \Delta
\end{equation}
and initial condition $\rho_0 dx$.
Observe that the generator \eqref{tdg} describing the law of large
numbers for the empirical measures reduces to \eqref{tdgzr} in the
zero range case; we remark however that in \eqref{tdgzr} the density $\rho_t$ is not, unless $\rho_t$ solves the hydrodynamic equation,
the density at time $t$ of the diffusion process. Furthermore, as proven in \cite{JLS}, when
$\rho$ is the solution to the hydrodynamic equation, the
diffusion with generator \eqref{tdgzr} is the limiting distribution of
a tagged particle in non-equilibrium.  

\begin{lemma}
  \label{t:qrvzr}
  For the zero range process $\mf F\colon \ms M_+\big(
  D([0,T];\bb T^d)\big)\to [0,+\infty]$ is infinite unless $R\in\ms M_+\big(
  C([0,T];\bb T^d)\big)$, with $(R_t)_{t\in [0,T]}
  =(\rho_t \, dx)_{t\in [0,T]}$,
  and in this case
  \begin{equation}\label{2form-zr}
    \mf F(R) = I_\mathrm{in}(\rho_0)
    + \Ent\big(R \,\big|\, Q(\rho) \,\big).
  \end{equation}
\end{lemma}

A representation of the QRV functional analogous to the one in the
previous lemma can be obtained whenever $D_\mathrm{s}(\rho)$ is
equal to $\sigma(\rho)/\rho$; however, apart from the
zero-range model, we are not aware of any model for which this
condition is fulfilled.
As discussed in Appendix~\ref{s:rap}, a conclusion similar to the
one stated above holds in general, but the reference measure $Q$ in
\eqref{2form-zr} does not depend only on the density $\rho$, the
single time marginal of $R$, but also on $j$, the expected current under $R$.

The representation of the rate function provided by
Lemma~\ref{t:qrvzr} clearly characterizes the zero level set of
$\mf F$ as the law of the non-linear Markov diffusion obtained as law
of large numbers for the empirical measure.  Furthermore, if we
consider the case of $N=\ell^d$ non-interacting particles whose
initial distribution is specified by sampling independently from the
uniform probability on $\bb T^d_\ell\cap\bb Z^d$ then
$I_\mathrm{in}(R_0)= \Ent(R_0|dx)$ so that $ \mf F(R) = \Ent(R|W )$
where $W$ is the law of a stationary Brownian motion on $\bb T^d$ with
diffusion coefficient $2$.  Readily, this statement follows directly
from Sanov's theorem and the convergence of the diffusively rescaled
random walk to the Brownian motion.

\begin{proof}
  Let $P(R)\in \ms M_+\big(C\big([0,T];\bb T^d)\big)$ be as defined
  around \eqref{tdgp}. By chain rule
  \begin{equation}\label{chain}
    \Ent\big( R \, |\, P(R)\,\big) =
    \Ent\big( R \, |\, Q(\rho) \,\big) 
    - \int\! dR\, \log \frac {d P (R)}{d  Q(\rho)},
  \end{equation}
  where $R_t=\rho_t dx$, $t\in[0,T]$.
  Recalling \eqref{rfh} and \eqref{rfqrv}, it is therefore enough to
  show
  \begin{equation*}
    \int\! dR\, \log \frac {d P (R)}{d  Q(\rho)}
      =  I_\mathrm{dyn}(\rho,j)
  \end{equation*}
  where $j=j^R$ is the representative of the expected current under $R$.

  By Girsanov formula and few direct computations to write the
  exponential martingale in terms of the Stratonovich integral, 
  \begin{equation}\label{Gir}
    \log \frac {d P (R)}{d  Q(\rho)}=
    \int_0^T\!
    \Big[ E_t(X_t) \circ dX_t
    - 
    \Big( \frac{\phi(\rho_t)}{\rho_t} \nabla\cdot E_t
    +      \frac{\phi(\rho_t)}{\rho_t} |E_t|^2 \Big)(X_t) \, dt
    \Big].
  \end{equation}
  By using that $\rho$ is the density of the single time marginals of
  $R$ and $j$ is the representative of the expected current,
  integrating by parts we deduce
  \begin{equation*}
     \int\! dR\, \log \frac {d P (R)}{d  Q(\rho)}
     = \int_0^T\!dt\!\int\!dx\, \Big[
     j_t\cdot E_t + \nabla \phi(\rho_t) \cdot E_t
     - \phi(\rho_t) |E_t|^2 \Big].
   \end{equation*}
  Recalling that $E=E^R$ has to be chosen so that $j= - \phi'(\rho)\nabla\rho
  +2 \phi(\rho) E$ we conclude the proof.
\end{proof}
We notice that, since the Radon-Nikodym derivative \eqref{Gir} is the
sum of two terms one depending on the empirical current and the other
by just the marginals of $R$, we have that $\int \!dR\,\log \frac {d P
  (R)}{d Q(\rho)}=\int\! dP(R)\,\log \frac {d P (R)}{d Q(\rho)}$ so
that the chain identity \eqref{chain} becomes the additivity property
\begin{equation}\label{chain2}
\Ent\big( R \, |\, P(R)\,\big) + \Ent\big( P(R) \, |\, Q(\rho)\,\big)=
\Ent\big( R \, |\, Q(\rho) \,\big) 
\end{equation}

\subsection*{The QRV functional for mean field interacting
  Brownians} 
 
In order to provide further support to the universality of the QRV
functional, we here consider the case of mean field interacting
Brownians.
Although this model is quite far from the context of hydrodynamical
scaling limits, we define a suitable cost functional by adapting the
recipe in \cite{QRV}. We then show that this functional is in fact
equal to the rate function for the empirical measure derived in
\cite{BDF,CL,Fi}.

We consider a family of $N$ diffusions on $\bb T^d$ evolving
according the stochastic equations
\begin{equation*}
  dX^i_t = -\frac 1N \sum_{j\neq i} \nabla U (X^j_t-X^i_t) dt + \sqrt{2}
  \, dW^i_t, \qquad i=1,\ldots, N
\end{equation*}
where $W^i$, $i=1,\ldots,N$ are independent standard Brownians on
$\bb T^d$ and $U\colon \bb T^d \to \bb R$ is a smooth pair potential.
Clearly, this Markovian dynamics is reversible with respect to the
mean field canonical Gibbs measure with $N$ particles and pair
potential $U$. 
In this context, the empirical density, current, and measure are
defined as in \eqref{ed}, \eqref{ec1}--\eqref{ec2}, \eqref{em},
with the prefactor $\ell^{d}$ replaced by $N$ and dropping the
diffusive rescaling.

The large deviations principle for the empirical density is a classical
result by Dawson and G\"artner \cite{DG} while the joint
large deviations for the empirical density and current have been
obtained in \cite{BC,O}. Referring to those articles for the required
conditions on the initial distribution of the particles, 
the corresponding dynamical rate function
$I_\mathrm{dyn}\colon C\big([0,T];\ms M_1(\bb T^d))\times \mf J\to
[0,+\infty]$ is given by
\begin{equation}
  \label{ldmf}
    I_\mathrm{dyn}(\rho,j) =\frac 14 \int_0^T\!dt\int\!dx\,
  \frac { | j+ \nabla\rho +\rho\, \rho*\nabla U |^2}{\rho}
\end{equation}
in which $\rho$ and $j$ are the representatives of $\pi$ and $\ms J$
and $*$ denotes space convolution. 
The interpretation of \eqref{ldmf} in terms of an applied field is as
in the hydrodynamical case in which the perturbed law of large numbers
reads
\begin{equation}
  \label{phemf}
  \begin{cases}
    \partial_t \rho +\nabla\cdot j =0, \\
    j= -\nabla\rho +\rho( 2E - \rho*\nabla U) .
  \end{cases}
\end{equation}

In order to introduce the QRV functional in this context, given 
the probability
$R\in \ms M_1\big(C\big([0,T];\bb T^d)\big)$ we still denote by
$(\rho_t)_{t\in[0,T]}$ the density of its single time marginals and
by $j=j^R$ the representative of the expected current under $R$. 
We then denote by $P(R)\in \ms M_1\big(C\big([0,T];\bb T^d)\big)$ the law of the
time-inhomogeneous diffusion with time dependent generator
\begin{equation*}
    L^R_t = \Delta + (2E^R_t-\rho_t*\nabla U)\cdot \nabla 
\end{equation*}
and initial condition $\big(P(R)\big)_0 =R_0$. Here the external field $E^R$ is
determined from $(\rho,j)$ by \eqref{phemf}.
The  QRV functional $\mf F$ for mean field interacting Brownians is
then defined as in \eqref{rfqrv} where $I$ is given by \eqref{rfh} with
$I_\mathrm{dyn}$ now given by \eqref{ldmf}.

Analogously to the zero-range case, given the family of probability
densities  $\rho=(\rho_t)_{t\in[0,T]}$, we denote by
$Q(\rho)\in \ms M_1\big(C\big([0,T];\bb T^d)\big)$ the
law of the time-inhomogeneous diffusion on $\bb T^d$ with time
dependent generator 
\begin{equation*}
  L^\rho_t =  \Delta -\rho_t*\nabla U \cdot \nabla
\end{equation*}
and initial condition $\rho_0 dx$.
The proof of the following statement is achieved by the same
computations as in Lemma~\ref{t:qrvzr} and it is therefore omitted. 

\begin{lemma}
  \label{t:qrvib}
  Let $\mf F\colon \ms M_1\big( C([0,T];\bb T^d)\big)\to [0,+\infty]$
  be the QRV functional for mean field interacting Brownians.  
  Then $\mf F(R)$ is infinite unless $(R_t)_{t\in [0,T]}=(\rho_t \,
  dx)_{t\in [0,T]}$, and in this case  
  \begin{equation}
    \label{imf}
    \mf F(R) = I_\mathrm{in}(\rho_0)
    + \Ent\big(R \,\big|\, Q(\rho) \,\big).
  \end{equation}
\end{lemma}

As proven in \cite{CL}, the right hand side of \eqref{imf} is the
large deviations rate function of the empirical measure for mean field
interacting Brownians. The previous lemma thus shows that the QRV
recipe can be applied also in this context.

For short-range interacting Brownians the hydrodynamical limit is
proven in \cite{V}, the converge of a tagged particle, for $d\ge 2$,
is analyzed in \cite{osada1,osada2} where the self-diffusion coefficient is
identified.  As discussed in \cite{Sp}, a peculiarity of this model is
that the mobility is given by $\sigma(\rho)=\rho$, as for
non-interacting particles. It does not however appear possible to
deduce a simpler expression for the QRV functional and one should
resort to the general formulation in \eqref{rfqrv}.
Furthermore, for $d=1$, it is possible to consider the limiting case
in which the interaction range vanishes, obtaining Brownians
interacting by partial reflection.
In this case we simply have $D_\mathrm{h}=1$,
$D_\mathrm{s}=\lambda(\lambda+\rho)^{-1}$, where $\lambda\ge 0$ is the
parameter quantifying the partial reflection, and
$\sigma(\rho)=\rho$.
As already mentioned, the large deviations principle for the empirical
measure is proven in \cite{Seo}, the corresponding rate function is
expressed in the QRV form.
We finally mention that the hydrodynamical large deviations for a
colored version of the stirring process have been discussed in
\cite{Van}; by projective limit argument in \cite{QRV}, it should then
be possible to obtain the large deviation rate for the empirical
measure. Since in the stirring process jumps occur at the same time,
it is not clear if such rate functions has the form \eqref{rfqrv}.

\section{Standard Schr\"odinger problem for interacting particles}
\label{s:spip}

According to its statistical interpretation, the standard
Schr\"odinger problem for interacting particles on the time interval
$[0,T]$ can be formulated as an optimization problem either for the
hydrodynamical functional
$I\colon C\big([0,T];\ms M_+(\bb T^d)\big)\times \mf J \to
[0,+\infty]$ or for the QRV functional
$\mf F\colon \ms M_+\big( C([0,T];\bb T^d)\big)\to [0,+\infty]$ as
respectively introduced in \eqref{rfh} and \eqref{rfqrv}.
To be precise, as large deviation rate functions for stochastic
lattice gases, these functionals are respectively defined on
$D\big([0,T];\ms M_+(\bb T^d)\big)\times \mf J$ and
$\ms M_+\big( D([0,T];\bb T^d)\big)$ but we can restrict to continuous
paths.  These functionals depend on the hydrodynamical transport
coefficients $D_\mathrm{h}$, $\sigma$ and on the self-diffusion
$D_\mathrm{s}$ that -- in principle -- can be derived from the
underlying microscopic dynamics. For the present purposes we instead
consider them to be the ``input data'' for the optimization problems
we are going to introduce.  We assume these coefficients are smooth
and that the Einstein relationship \eqref{er} holds for a strictly
convex free energy $f$.  Recall that for
$R\in \ms M_+\big( C([0,T];\bb T^d)\big)$ the family of its single
time marginals is denoted by $(R_t)_{t\in[0,T]}$.

\begin{definition}
  \label{t:ssp}
  Fix $T>0$ and $\mu_0,\mu_1\in \ms M_+(\bb T^d)$ which we assume to
  be absolutely continuous with respect to $dx$ and with the same
  mass.
  \begin{itemize}
  \item [(HSP)] The \emph{hydrodynamical Schr\"odinger problem}
    is the optimization problem 
    \begin{equation*}
      V_T(\mu_0,\mu_1) :=
      \inf \big\{ I(\pi,\ms J),\; \pi_0=\mu_0,\, \pi_T=\mu_1 \big\}.
    \end{equation*}
   \item [(MSP)] The \emph{measure Schr\"odinger problem}
    is the optimization problem 
    \begin{equation*}
      \ms V_T(\mu_0,\mu_1)
      :=  \inf \big\{ \mf F(R),\; R_0=\mu_0,\, R_T=\mu_1 \big\}.
    \end{equation*}   
  \end{itemize}
\end{definition}

As next detailed, the optimization (MSP) is a generalization to
interacting particle for the classical Schr\"odinger problem for a
Brownian motion on $\bb T^d$.
The optimization (HSP) is exactly the version of the Schr\"odinger
problem introduced in \cite{CC}. As we have fixed the initial datum
$R_0=\pi_0=\mu_0$, the contribution from the initial condition
$I_\mathrm{in}$ in \eqref{rfh} is not relevant for either (HSP) or (MSP)
and can be dropped.   
By the goodness of the functionals $I$ and $\mf F$, if
$V_T(\mu_0,\mu_1)<+\infty$, respectively
$\mc V_T(\mu_0,\mu_1)<+\infty$, then the infimum in (HSP),
respectively in (MSP), is achieved, i.e.\ it is actually a minimum.
In contrast to the case of independent particles, as neither $I$ nor
$\mf F$ is in general convex, uniqueness of the minimizer for (HSP)
and (MSP) may fail.

If we consider the case in which $\mu_0=m \, dx$, $m\in(0,\infty)$, so
that that the corresponding density is a stationary solution to the
hydrodynamical equation \eqref{nde}, and $T=\infty$ then the
optimization (HSP) defines, in the Freidlin-Wentzell terminology
\cite{FW}, the \emph{quasi-potential} associated to the rate function $I$.
In the present setting of reversible dynamics, $V_{\infty}(mdx,\cdot)$
is then given by the free energy functional on the right hand side of
\eqref{Iincan}, which describes the static large deviations of the
empirical density when the particle are sampled according to the Gibbs
measure.  On the other hand, when the underlying microscopic system is
not in equilibrium the stochastic dynamics in no longer reversible and
the computation of the quasi-potential becomes a most relevant issue,
we refer to \cite{BDGJLrev} for a review.  In particular, in the case
of a particular boundary driven one-dimensional model, when the
density of $\mu_0$ is the stationary solution to the hydrodynamical
equation there are examples of $\mu_1$, as constructed in
\cite{BDGJLlpt}, for which the minimizer for $V_{\infty}(\mu_0,\mu_1)$
is not unique. As these values of $\mu_1$ correspond to points in
which the tangent functional to the quasi-potential is not unique,
they can be interpreted as the occurrence of phase transitions.

Let $I^{(1)}\colon C\big([0,T];\ms M_+(\bb T^d)\big)\to [0,+\infty]$
be the projection on the density of the hydrodynamical rate function,
\begin{equation}
  \label{Iden}
  I^{(1)}(\pi) := \inf\big\{  I(\pi,\ms J),\; \ms J \textrm{ such
    that }\partial_t \pi
  +\nabla\cdot \ms J = 0\big\}.
\end{equation}
The optimization (HSP) can then be recast as
\begin{equation}\leqnomode
  \tag{HSP'}
      V_T(\mu_0,\mu_1) =
      \inf \big\{ I^{(1)}(\pi),\; \pi_0=\mu_0,\, \pi_T=\mu_1 \big\}.
\end{equation}
As discussed in \cite{BDGJLrev,CC}, the functional $I^{(1)}$ can be
regarded as an action functional; while the associated Lagrangian is
somewhat cumbersome to write, the corresponding Hamiltonian simply
reads
\begin{equation}
  \label{ham}
  \ms H(\rho,H) = \int\! dx\,
  \big[ \sigma(\rho) |\nabla H|^2 - \nabla H
  \cdot D_\mathrm{h}(\rho)\nabla \rho \big] 
\end{equation}
where $H\colon \bb T^d\to \bb R$ is the momentum conjugated to the
density $\rho$.  
In view of this Hamiltonian structure, if
$(\pi^*_t)_{t\in[0,T]} = (\rho^*_t dx)_{t\in[0,T]}$ is a minimizer for
(HSP') then $\rho^*= (\rho^*_t)_{t\in[0,T]}$ solves the canonical
equations associated to the Hamiltonian $\ms H$,
\begin{equation}
  \label{caneq}
  \begin{cases}
    \partial_t\rho = \nabla\cdot \big(D_\mathrm{h}(\rho)\nabla\rho\big)
    - 2\, \nabla\cdot \big(\sigma(\rho) \nabla H \big),\\
    \partial_t H = - D_\mathrm{h}(\rho)\Delta H
    - \sigma'(\rho) |\nabla H|^2,
  \end{cases}
\end{equation}
for a suitable conjugated momentum $H^*= (H^*_t)_{t\in[0,T]}$.
We refer to \cite{CC} for the actual proof of this statement.

\subsection*{Independent particles}
Before analyzing the problems (HSP) and (MSP) in the context of
interacting particles, we next discuss the case of independent
particles. In particular, referring to \cite{Lrev} for the proofs of
the following claims and for a more detailed exposition, we review the
standard Schr\"odinger problem for a Brownian on $\bb T^d$.  This
special situation corresponds to the case of a microscopic dynamics
given either by independent Brownians or a stochastic lattice gas with
non-interacting particles.

The cost functionals are specified by the coefficients
$D_\mathrm{h}=D_\mathrm{s}=1$ and $\sigma(\rho)=\rho$.  Accordingly,
the dynamical contribution to the hydrodynamical rate function is
\begin{equation}\label{phe-I-ind}
  I_\mathrm{dyn}(\rho,j) = \int_0^T\!dt\int\!dx\,
  \rho |E|^2.
\end{equation}
where $E=E_t(x)$ is computed from $(\rho,j)$ by \eqref{phe} that for
independent particles reads $j=-\nabla\rho +2\rho E$.
The QRV rate functional $\mf F$ is then given by \eqref{rfqrv} in
which the time dependent generator \eqref{tdgp} becomes simply
$L^R_t=\Delta+2 E^R_t\cdot \nabla$ with $E^R$ as before with $j=j^R$ the
representative of the expected current under $R$. 

In view of Lemma~\ref{t:qrvzr}, the functional $\mf F$ admits the
alternative representation 
\begin{equation*}\label{2form-zr-ind1}
  \mf F(R) = I_\mathrm{in}(\rho_0)
  + \Ent\big(R \,\big|\, W(\rho_0) \,\big)\,,
\end{equation*}
where $W(\rho_0)$ is the law of a Brownian motion on $\bb T^d$ with
diffusion coefficient $2$ and initial distribution $\rho_0\, dx$.
Indeed, for independent particles the generator in \eqref{tdgzr} does not
depend on $\rho$ and it simply $L^\rho=\Delta$; the dependence on
$\rho$ is therefore just on the initial condition $\rho_0$. Formula \eqref{2form-zr-ind} is \eqref{2form-zr}
with $Q(\rho)=W(\rho_0)$.
If we further consider the case in which the particles at time zero
are distributed as i.i.d.\ random variables with a common law
$\rho_\mathrm{in}\,dx$, then 
$I_\mathrm{in}(\rho_0)=\Ent(\rho_0|\rho_\mathrm{in}\,dx)$ so that
\begin{equation}\label{2form-zr-ind}
  \mf F(R) = \Ent\big(R \,\big|\, W(\rho_\mathrm{in}) \,\big)\,.
\end{equation}
In the special case in which $\rho_\mathrm{in}=1$, particles are
initially uniformly distributed on $\bb T^d$ and we denote the
reference measure by $W$, the law of a stationary Brownian with
diffusion $2$ on $\bb T^d$.
Then \eqref{2form-zr-ind} corresponds to the original Schr\"{o}dinger
problem \cite{Lrev}. In this case, as the reference measure
$W$ is reversible, it has some remarkable features that we shortly
recall.

For $R\in \ms M_1\big( C([0,T];\bb T^d)\big)$
we denote by $R_{0T}\in \ms M_1(\mathbb T^d\times \mathbb T^d)$ the
joint law of its marginals at the times $0$ and $T$.
In particular, $W_{0T}=p_T(x,y)\,dxdy$ where $p_T(x,y)$ is the heat
kernel on $\bb T^d$, denotes the joint law of the marginals at time
$0$ and $T$ of a stationary Brownian.
For $\wp \in \ms M_1(\bb T^d\times \bb T^d)$, we denote by
$\wp_0$ and $\wp_1$ its marginals.

Fix $\mu_0,\mu_1\in \ms M_1(\bb T^d)$ which we assume to be absolutely
continuous with respect to $dx$. We introduce the \emph{static
Schr\"{o}dinger problem} (SSP) as the optimization 
\begin{equation}\leqnomode
  \tag{SSP}      
  \mc V_T^\mathrm{s}(\mu_0,\mu_1) :=
  \inf \big\{ \Ent(\wp |W_{0T}),\; \wp_0=\mu_0,\,
  \wp_1=\mu_1 \big\}.
\end{equation}

In view of the strict convexity of $\Ent(\cdot|W_{0T})$, (SSP) admits
a unique minimizer that is furthermore given by 
\begin{equation*}
  \wp^* =p_T(x,y)e^{f(x)+g(y)}\, dxdy  
\end{equation*}
where the functions $f,g\colon \bb T^d\to \bb R$ are determined by
$\wp^*_0=\mu_0$, $\wp^*_1=\mu_1$.  
Moreover, the value function for (HSP), (MSP), and (SSP)
coincide, i.e.\
\begin{equation*}
  \mc V_T^\mathrm{s}(\mu_0,\mu_1)
   =\ms V_T(\mu_0,\mu_1) =V_T(\mu_0,\mu_1). 
\end{equation*}

The unique minimizer for (MSP) can be obtained from $\wp^*$, the
unique minimizer for (SSP). More precisely, letting 
$W_{[0,T]}^{x,y}$ be the law of the Brownian bridge pinned at $x$ at
time zero and at $y$ at time $T$, then the minimizer for (MSP) is
\begin{equation}\label{Rstar}
  R^* = \int\! \wp^*(dxdy)\, W^{x,y}_{[0,T]}
\end{equation}
in which we observe that $R^*_{0T} =\wp^*$ and that $R^*$ is
Markovian. In particular it is a Brownian bridge with laws $\mu_0$ and $\mu_1$ respectively
at time $0$ and $T$ and it is a diffusion process related to the minimizer of the (HSP) as now we detail.

In order to describe the the minimizer for (HSP), we first note that
for independent particles the canonical equations \eqref{caneq} read
\begin{equation}
  \label{caneq-ind}
  \begin{cases}
    \partial_t\rho = \Delta \rho
    - 2\, \nabla\cdot \big(\rho \nabla H \big),\\
    \partial_t H = - \Delta H - |\nabla H|^2.
  \end{cases}
\end{equation}
The minimizer for (HSP) is then given by $(\rho^*,-\nabla \rho^*
+2\rho^*\nabla H^*)$ where $(\rho^*,H^*)$ is the solution to
\eqref{caneq-ind} meeting the conditions 
$\rho_0\,dx =\mu_0$, $\rho_T\,dx=\mu_1$.

The minimizer for the (MSP) is related to to minimizer for the (HSP)
by the fact that $R^*$ in \eqref{Rstar} is the law of the diffusion
process with generator $\Delta +2E^*\cdot \nabla$ and initial
condition $\mu_0$.

The differential problem \eqref{caneq-ind} can be solved by a
symplectic transformation, the classic Cole-Hopf transformation, that transforms the equations into two
decoupled heat equations \cite{KMS}. The symplectic transformation
from the variables $(\rho, H)$ into the new variables $(\xi,\eta)$ is
given by
\begin{equation}\label{C-H}
  \begin{cases}
    \xi=\rho \, e^{-H},\\
    \eta=e^H.
  \end{cases}
\end{equation}
Recalling \eqref{ham}, the new Hamiltonian reads
\begin{equation*}
  \ms K(\xi,\eta)=\ms H(\rho(\xi,\eta),H(\xi,\eta))=-\int\!dx\, \nabla \xi\cdot \nabla \eta
\end{equation*}
whose corresponding canonical equations are decoupled
\begin{equation*}
  \begin{cases}
    \partial_t \xi = \Delta \xi,\\
    \partial_t \eta = - \Delta \eta.
  \end{cases}
\end{equation*}
Finally, the boundary conditions become 
$\xi_0\eta_0\, dx=\mu_0$, $\xi_T\eta_T\, dx=\mu_1$.

The coincidence $\ms V_T(\mu_0,\mu_1)= V_T(\mu_0,\mu_1)$ of the minimal values for the variational problems (HSP) and (MSP), gives the formula \cite[Prop.~6]{Lrev} that is an analogous of the Benamou-Brenier formula for the Schr\"odinger problem. We first observe
that, as follows from \eqref{2form-zr-ind},
\begin{equation}\label{ms-ind}
  \ms V_T(\mu_0,\mu_1) =
  \inf \big\{ \Ent(R|W),\; R_0=\mu_0,\, R_T=\mu_T \big\}.
\end{equation}
Using formula \eqref{phe-I-ind} we have
\begin{equation}
  \label{bbfor}
  V_T(\mu_0,\mu_1)=\inf \int_0^T\!dt\int\!dx\,
  \rho(x,t) |E(x,t)|^2,
\end{equation}
where the infimum is over all $(\rho, E)$ satisfying
$\partial_t\rho=\Delta \rho-2\nabla\cdot(\rho E)$ and meeting the
boundary conditions $\rho_0\, dx =\mu_0$, $\rho_T\, dx =\mu_1$. We
have that the right hand side of \eqref{ms-ind} is equal to the right
hand side of \eqref{bbfor} and this is the formula derived in
\cite[Prop.~6]{Lrev}.

\subsection*{Interacting particles}

As we next show, minimizers for (MSP) can be completely characterized
in terms of minimizers for (HSP') also for interacting particles. This
is due to the specific form of the rate functional and also to the
fact that we are considering constraints just on hydrodynamic
observables.  In contrast to the case of independent particles, it
does not however appear possible to reduce (MSP) to a static problem
analogous to (SSP).

\begin{theorem}
  \label{t:dspvsmsp}
  For each $T>0$ and $\mu_0,\mu_1$ we have
  $V_T(\mu_0,\mu_1)=\ms V_T(\mu_0,\mu_1)$.\hfill\break
  Let also $\pi^*$ be a minimizer for \emph{(HSP')} and denote by
  $\rho^*= (\rho^*_t)_{t\in[0,T]}$ the corresponding density with
  conjugated momentum $H^*= (H^*_t)_{t\in[0,T]}$ that we recall solve the system of equations \eqref{caneq}.
  Then a minimizer for \emph{(HSP)} is $(\pi^*,\ms J^*)$ where $\ms J^*$ is
  represented by $-D_\mathrm{h}(\rho^*)\nabla\rho^* +
  2\sigma(\rho^*)\nabla H^* $.
  Furthermore, a minimizer to \emph{(MSP)} is provided by $R^*$, the law of
  the time-inhomogeneous Markov diffusion with time dependent
  generator
  \begin{equation}
    \label{tdgpop}
    L^*_t = \nabla\cdot D_\mathrm{s} (\rho^*_t) \nabla 
    +
    \Big\{
    \big[ D_\mathrm{s} (\rho^*_t) -  D_\mathrm{h} (\rho^*_t)\big]
    \frac {\nabla \rho^*_t}{\rho^*_t} 
    + 2 \frac {\sigma(\rho^*_t)}{\rho^*_t} \nabla H^*_t \Big\}\cdot \nabla
  \end{equation}
  and initial condition $R^*_0 =\mu_0$.\hfill\break
  Conversely, if $R^*$ is a minimizer to \emph{(MSP)} then $R^*$ is Markovian,
  $R^*=P(R^*)$ and the family of its 
  single time marginals $(R^*_t)_{t\in[0,T]}$ is a minimizer to \emph{(HSP')}.
  Furthermore, letting $R^*_t=\rho^*_t \,dx$, $t\in[0,T]$, the
  external field $E^*=E^{R^*}$ associated to $R^*$ has the form
  $E^* =\nabla H^*$ and the pair $(\rho^*,H^*)$ solves weakly the
  canonical equations \eqref{caneq}. Finally  $(\pi^*,\ms J^*)$, where $\pi^*$ has density $\rho^*$ and $\ms J^*$ is
  represented by $-D_\mathrm{h}(\rho^*)\nabla\rho^* +
  2\sigma(\rho^*)\nabla H^* $, is a minimizer for \emph{(HSP)}.
\end{theorem}

In terms of the underlying interacting particle system, the minimizer
$R^*$ can be obtained as law of large numbers for the empirical
measure by considering a weakly asymmetric perturbation with time
dependent external field $2 \nabla H^*$.

Theorem~\ref{t:dspvsmsp} together with the characterization of the QRV
functional for the zero range process in Lemma~\ref{t:qrvzr} implies
the following generalization of \cite[Prop.~6]{Lrev} as recalled in
\eqref{bbfor}.
Fix $\phi\colon \bb R_+\to\bb R_+$ smooth, strictly increasing, and
satisfying $\phi(0)=0$. Given $m\in (0,+\infty)$ and $R\in \ms
M_m\big(C([0,T];\bb T^d)\big)$, assume $R_t=\rho_t\,dx$ and denote by
$Q(\rho)\in \ms M_m\big(C([0,T];\bb T^d)\big)$ the law of the
time-inhomogeneous diffusion on $\bb T^d$ with time dependent
generator \eqref{tdgzr} and initial condition $\rho_0 dx$.

\begin{corollary}
  \label{t:gbb}
  Let $\mu_0,\mu_1\in \ms M_m(\bb T^d)$ be absolutely continuous with
  respect to $dx$. Then
  \begin{equation*}
    \inf\big\{ \Ent\big(R|Q(\rho)\big),\; R_0=\mu_0,\,R_T=\mu_1 \big\}
    = \inf \; \int_0^T\!dt\int\!dx\,\phi(\rho_t) |E_t|^2
  \end{equation*}
  where the infimum on the right hand side is over all $(\rho, E)$
  satisfying $\partial_t\rho=\Delta \phi(\rho) -2\nabla\cdot( \phi(\rho) E)$ and
  meeting the boundary conditions $\rho_0\, dx =\mu_0$,
  $\rho_T\, dx =\mu_1$.
\end{corollary}

In view of Lemma~\ref{t:qrvib}, an analogous statement holds also in
the case of mean field interacting Brownians.

\begin{proof}[Proof of Theorem~\ref{t:dspvsmsp}]
  As follows from definition \eqref{rfqrv}, $\ms
  V_T(\mu_0,\mu_1) \ge  V_T(\mu_0,\mu_1)$ with equality if and only if
  $R$ is equal to the Markovian measure $P(R)$ described around
  \eqref{tdgp}. To complete the proof it is therefore enough to show that if
  $(\pi^*,\ms J^*)$ is a minimizer to (HSP) then
  $j^* = -D_\mathrm{h}(\rho^*)\nabla\rho^* + 2\sigma(\rho^*)\nabla
  H^*$ where $(\rho^*,j^*)$ is the representative of  $(\pi^*,\ms
  J^*)$ and $H^*$ is the momentum conjugated to $\rho^*$.
  To this end, we set
  \begin{equation*}
    \hat \jmath :=
    -D_\mathrm{h}(\rho^*)\nabla\rho^* + 2\sigma(\rho^*)\nabla H^*
  \end{equation*}
  and write $j^*=\hat \jmath+ \tilde \jmath$.
  Since $\partial_t\rho^* +\nabla\cdot j^*=0$ and the pair
  $(\rho^*,H^*)$ solves the first equation in \eqref{caneq}, we deduce
  $\nabla\cdot \tilde \jmath =0$. Hence, integrating by parts,
  \begin{equation*}
    I_\mathrm{dyn}(\rho^*,j^*)
    = I_\mathrm{dyn}(\rho^*,\hat \jmath)
    + \frac 14 \int_0^T\!dt\!\int\!dx \, 
    \frac { | \tilde \jmath |^2}{\sigma(\rho^*)}
  \end{equation*}
  which yields $j^*=\hat \jmath$.
\end{proof}

In Theorem~\ref{t:dspvsmsp} we have characterized the minimizers to
(MSP) as forward Markov measures on $C([0,T];\bb R^d)$. Clearly, it also
possible to characterize them as \emph{backward} Markov measures; this
will provide a probabilistic interpretation for the dual potentials
introduced in \cite{CC}.
  
\begin{lemma}
  \label{t:trR*}
   Let $\theta$ be the involution on $C([0,T];\bb T^d)$ defined by
   $(\theta X)_t=X_{T-t}$.
   Given a minimizer $R^*$ for \emph{(MSP)}, set $\hat R:= R^*\circ \theta$ and
   let $\hat H\colon [0,T]\times \bb T^d\to \bb R$ be such that
   \begin{equation}\label{Born}
     H_t^*+\hat H_{t} = f'(\rho^*_t), \qquad t\in[0,T]
   \end{equation}
   where $\rho^*$ is the density of the single time marginals of $R^*$
   while $H^*$ is the conjugated momentum.
   Then $\hat R$ is the law of the time-inhomogeneous Markov diffusion
   with time dependent generator
  \begin{equation}
    \label{tdgpopr}
    \hat L_t = \nabla\cdot D_\mathrm{s} (\rho^*_{T-t}) \nabla 
    +
    \Big\{
    \big[ D_\mathrm{s} (\rho^*_{T-t}) -  D_\mathrm{h} (\rho^*_{T-t})\big]
    \frac {\nabla \rho^*_{T-t}}{\rho^*_{T-t}} 
    + 2 \frac {\sigma(\rho^*_{T-t})}{\rho^*_{T-t}}
    \nabla {\hat H}_{T-t} \Big\}\cdot \nabla
  \end{equation}
  and initial condition $\mu_1$. 
\end{lemma}

As shown in \cite{CC}, the pair $(H^*,\hat H)$ solves
\begin{equation}\label{secH}
  \begin{cases}
    \partial_t H^* = - D_\mathrm{h}(\rho^*)\Delta H^* - \sigma'(\rho)
    \big|\nabla H^*\big|^2,\\
    \partial_t \hat H = D_\mathrm{h}(\rho^*)\Delta \hat H +
    \sigma'(\rho)\big|\nabla \hat H\big|^2.
  \end{cases}
\end{equation}
The validity of the above pair of equations can also be deduced
directly by observing that both 
$(\rho^*_t,H^*_t)$
and $(\rho^*_{T-t}, \hat H_{T-t})$
solve the Hamilton equations \eqref{caneq}.
The two equations in \eqref{secH} are indeed two copies of the second
equation in \eqref{caneq} with the same $\rho^*$, one of the two with
inverted sign.

\begin{proof}
  According to the general result in \cite{HP}, the measure
  $\hat R$ is the law of the time-inhomogeneous Markov diffusion
   with time dependent generator
   \begin{equation*}
     \begin{split}
     \hat L_t =  & \nabla\cdot D_\mathrm{s} (\rho^*_{T-t}) \nabla 
    +
    \Big\{
    - \big[ D_\mathrm{s} (\rho^*_{T-t}) -  D_\mathrm{h} (\rho^*_{T-t})\big]
    \frac {\nabla \rho^*_{T-t}}{\rho^*_{T-t}} 
    - 2 \frac {\sigma(\rho^*_{T-t})}{\rho^*_{T-t}}
    \nabla {H^*}_{T-t}
    \\
    & \phantom{\nabla\cdot D_\mathrm{s} (\rho^*_{T-t}) \nabla +}
      - 2 \,\nabla\cdot D_\mathrm{s} (\rho^*_{T-t})
  +\frac 2{\rho^*_{T-t}}
  \nabla\cdot \big(D_\mathrm{s} (\rho^*_{T-t}) \rho^*_{T-t} \big) 
    \Big\}\cdot \nabla.
  \end{split}
\end{equation*}
In view of the Einstein relationship \eqref{er}, straightforward
computations yield the statement.
\end{proof}
Also in this case the generator \eqref{tdgpopr} can be characterized
as the unique diffusion process \eqref{gen-diff} having time marginal
$\rho^*_{T-t}$, averaged current $-j^*_{T-t}$ and time dependent
diffusion coefficient $D(t)=D_\mathrm{s} (\rho^*_{T-t})$.

\smallskip
To further compare the previous statement to the case of independent
particles, we write equation \eqref{Born} in the form
\begin{equation}\label{veraborn}
\rho^*_t=(f')^{-1}\big(H^*_t+\hat H_t\big),
\end{equation}
that is well defined since $f'$ is monotone increasing.
As for independent particles $f'(\rho)=\log\rho$,
\eqref{veraborn} reduces to what mathematicians
call Euclidean Born's formula, see \cite[Thm.~3.3]{Lrev}.

\subsection*{Relationship with completely integrable systems}
 
For a special class of particle systems, that corresponds to the case
in which the hydrodynamic diffusion $D_\mathrm{h}$ is constant and the
mobility $\sigma$ is a second order polynomial, it is possible, in the
one dimensional case, to solve the variational problem (HSP') by a
change of variables that transforms the Hamiltonian equations
\eqref{caneq} into a completely integrable system \cite{MHS,MHSl}.
Examples of models in this class are the exclusion/inclusion process
and the KMP process \cite{BDGJLrev}; the quadratic mobility can be
either concave or convex.
In principle, this transformation can be exploited to obtain analytic
solutions to the the Schr\"odinger problems.

The transformations are a generalization of the canonical Cole-Hopf
transformation \eqref{C-H} used in the case of independent particles.
In the new variables, the equations have still the canonical form, but
the corresponding Hamiltonian is different from the original one.
In the case of a one dimensional domain that is the infinite line, the
transformations are \cite{MHS,MHSl}:
\begin{equation}\label{simply-trans}
\begin{cases}
\xi=\frac{1}{\sigma'(\rho)}\partial_x\left(\sigma(\rho)\exp\left[-\int_{-\infty}^x{\sigma'(\rho)\partial_y 
    H}\,dy\right]\right)\,, \\ 
\eta=\frac{-1}{\sigma'(\rho)}\partial_x\left(\exp\left[\int_{-\infty}^x{\sigma'(\rho)\partial_y H}\,dy\right]\right)\,.
\end{cases}
\end{equation}
After such transformation the canonical equations \eqref{caneq} become
in the new variables the Ablowitz-Kaup-Newell-Segur (AKNS)
equations,
\begin{equation}\label{kirone2}
\begin{cases}
\partial_t\xi= \partial^2_x\xi-2\eta\xi^2\,, \\
\partial_t\eta=-\partial^2_x\eta+2\xi \eta^2\,,
\end{cases}
\end{equation}
which can be studied using the inverse scattering method. The
equations \eqref{kirone2} are Hamiltonian with respect to the
Hamiltonian
$\ms K(\xi,\eta)=\int \left(\nabla \xi\cdot \nabla
  \eta-\xi^2\eta^2\right) dx$ that is different from
$\ms H(\rho(\xi,\eta),H(\xi,\eta))=\int \xi \eta dx$.

In the case of the KMP model, that corresponds to
$\sigma(\rho)=\rho^2$, the Hamilton equations \eqref{caneq} can also
be directly related to the nonlinear Schr\"odinger equation which is a
classic integrable system \cite{BSM}.

\section{Schr\"odinger problem with constraint on the expected current}
\label{s:spcec}

Within this section we consider versions of the Schr\"odinger problem
for interacting particles with constraints on the time averaged
expected current.  According to the physical interpretation of
interacting particle systems as basic models for out of equilibrium
processes, like mass transportation, such constraints appears natural
both from a conceptual point of view and in the description of an
actual experiment. The general picture is similar to the case of
constraints on the initial and final density discussed in
Section~\ref{s:spip}, in particular the optimal measure will be a
Markovian measure. On the other hand, the external field needed to
realize a minimizer will no longer be a gradient vector field. Again
we are considering here constraints only in the hydrodynamic
observables so that the minimal values of the variational problems can
be computed considering just the hydrodynamic part of the rate
functional.

While it is not here discussed, it is also natural to introduce
versions of the Schr\"odinger problem with constraints on the law of
the stochastic empirical current, as introduced in \eqref{def-tagged},
rather than on its expectation. In this case it is not possible to
reduce to the hydrodynamical rate function and, in general, the
optimal measure will not be Markovian.

Recall that $I$ and $\mf F$ are the hydrodynamical functional and the
QRV functional as respectively introduced \eqref{rfh} in and
\eqref{rfqrv}.  We here denote $\Omega_0$ the space of smooth
functions $f \colon \bb T^d\to \bb R$ and by $\Omega_1$ the space of
smooth functions $\omega\colon \bb T^d\to \bb R^d$. Note that, in
contrast to Section~\ref{s:rf}, $\omega\in\Omega_1$ does not depend on
time.

\subsection*{Constraint on the time averaged expected current}

The Schr\"odinger problem with constraints on the time averaged
expected current is formulated as follows in which we emphasize that
we consider the cost per unit of time.

\begin{definition}
  \label{t:spmc}
  Fix $T>0$ and a divergence free vector field $\bar J \in L^2(\bb
  T^d;\bb R^d)$
  \begin{itemize}
  \item[(HSPC)]The \emph{hydrodynamic Schr\"odinger problem with
      current constraint} is the optimization problem 
    \begin{equation*}
      \phi_T(\bar J) :=
      \frac 1T
      \inf \Big\{ I(\pi,\ms J),\; \frac 1T \ms J(\omega)
      = \int\!dx\, \bar J\cdot \omega,\; \omega \in\Omega_1\Big\}.
    \end{equation*}
   \item[(MSPC)]The \emph{measure Schr\"odinger problem with
      current constraint} is the optimization problem 
    \begin{equation*}
      \Phi_T(\bar J) :=
      \frac 1T
      \inf \Big\{ \mf F(R),\;
      \frac 1T\int\!R(dX) \,  \int_0^T\!\omega(X_t)\circ dX_t
      = \int\!dx\, \bar J\cdot \omega,\; \omega \in\Omega_1\Big\}.
    \end{equation*}   
  \end{itemize}
\end{definition}

As discussed in Section~\ref{s:rf}, if $I(\pi,\ms J)<+\infty$, then
$\ms J$ admits a $L^2$ representative $j=j_t(x)$. In terms of this
representative, the constraint in (HSPC) can be rewritten as
\begin{equation}
  \label{conrc}
   \frac 1T \int_0^T\!dt \, j_t = \bar J.
\end{equation}
Likewise, if $\mf F(R)<+\infty$ then $\int\!R(dX)\, \mc J(X)$ admits a
$L^2$ representative $j=j^R$ and the constraint in (MSPC) can be
rewritten as above.

In the context of the one-dimensional exclusion process the problem
(HSPC) has been originally introduced in \cite{BD1} in the case
$T=\infty$. The present variational formulation is due to 
\cite{BDGJLldcur} where it is proven the existence of
$\lim_{T\to\infty} \phi_T(\bar J)$.
We refer to \cite{BGL} for a variational formulation of the
corresponding limiting problem.

Since $\bar J$ is divergence free, there is a simple strategy to
achieve the constraint in (HSPC): consider the constant path
$\rho_t=\bar\rho$, $j_t =\bar J$, $t\in[0,T]$, and optimize on the
density $\bar\rho$.  The corresponding dynamical cost $\mc U (\bar J)$
is given by
\begin{equation*}
  \mc U(\bar J) =\inf_{\bar\rho}
  \frac 14 \int\!dx\,
  \frac { | \bar J+ D_\mathrm{h}(\bar\rho) \nabla\bar\rho |^2}
  {\sigma(\bar\rho)}.
\end{equation*}
Then
\begin{equation}
  \label{f<u}
  \phi(\bar J) := \lim_{T\to\infty} \phi_T(\bar J) \le \mc U(\bar J).
\end{equation}
As shown in \cite{BDGJLldcur} this strategy is optimal for the zero
range case, in particular for independent particles.  On the other
hand, in the one-dimensional case, as discussed in \cite{BD2} for the
weakly asymmetric simple exclusion and in \cite{BDGJLldcur} for the
KMP model, the inequality \eqref{f<u} is strict for $\bar J$ large
enough.  In this case the minimizer for the problem (HSPC), in the
limit $T\to\infty$, exhibits a non trivial time dependence.

As in the case of  the Schr\"odinger problem with constraint on the
initial and final distribution of the particles, see
Theorem~\ref{t:hspcmspc} below, the optimal measure for (HSPC) can be
characterized in terms of the optimal path for (MSPC).

\subsection*{Constraints on both the density and the expected current} 

We next introduce a version of the Schr\"odinger problem with
constraints both on the initial and final density and on the expected
current. Again we consider the cost per unit of time.

\begin{definition}
  \label{t:spmdc}
  Fix $T>0$ and $\mu_0,\mu_1\in \ms M_+(\bb T^d)$ which we assume to
  be absolutely continuous with respect to $dx$ and with the same
  mass. Fix also  $\bar J \in L^2(\bb T^d;\bb R^d)$ satisfying the
  compatibility condition 
  \begin{equation}
    \label{cc}
    \mu_1(f)-\mu_0(f) = T\int\!dx \, \bar J \cdot \nabla f,
    \qquad f\in\Omega_0.  
  \end{equation}
  \begin{itemize}
  \item[(HSPDC)]The \emph{hydrodynamic Schr\"odinger problem with
      density and current constrain\-ts} is the optimization problem 
     \begin{equation*}
       \psi_T(\mu_0,\mu_1; \bar J) :=
       \frac 1T \inf \Big\{ I(\pi,\ms J),\;\pi_0=\mu_0,\:\pi_T=\mu_1,\:
       \ms J(\omega)
       = \frac 1T\int\!dx\, \bar J\cdot \omega,\; \omega \in\Omega_1\Big\}.
     \end{equation*}
   \item[(MSPDC)]The \emph{measure Schr\"odinger problem with
      density and current constraint} is the optimization problem 
    \begin{equation*}
      \begin{split}
      \Psi_T(\mu_0,\mu_1; \bar J) :=&
      \frac 1T \inf \Big\{ \mf F(R),\;R_0=\mu_0,\:R_T=\mu_1,\:
      \\
      &\quad
      \frac 1T \int\!R(dX) \int_0^T\!\omega(X_t)\circ dX_t
      = \int\!dx\, \bar J\cdot \omega,\; \omega \in\Omega_1\Big\}.
    \end{split}
  \end{equation*}   
  \end{itemize}
\end{definition}

If we denote by $j=j_t(x)$ both the $L^2$ representative of $\ms J$
and of the expected current under $R$, the current constraint in
(HSPDC) and (MSPDC) can be rewritten as in \eqref{conrc}
\begin{equation*}
  \frac 1T \int_0^T\!dt \, j_t = \bar J
\end{equation*}
where we emphasize that we have now assumed
the compatibility condition \eqref{cc} instead of
$\nabla\cdot \bar J=0$.

As in the case of (HSP'), the optimization (HSPDC) can be formulated
in terms of an action functional. For this purpose, we recall the
extended Hamiltonian structure introduced in \cite{BDGJLrev}.  In the
present case in which the space domain is the torus, we need to
slightly modify the construction.
Fix a path $(\pi,\ms J)$ with finite $I$ cost and denote by
$(\rho_t,j_t)$, $t\in[0,T]$, the corresponding density and
representative of the current.  Let $A_0\colon \bb T^d\to \bb R^d$ be
such that
\begin{equation}\label{defA0}
\nabla\cdot A_0=\rho_0-m\,, \qquad m=\int \rho_0(x)\,dx\,.
\end{equation}
We introduce $A\colon [0,T]\times \bb T^d\to \bb R^d$ defined  by
\begin{equation}\label{defA}
  A_t = A_0 - \int_0^t\!ds\, j_s\,,
\end{equation}
that satisfies the relations 
\begin{equation}\label{Arj}
\left\{
\begin{array}{l}
\nabla\cdot A_t=\rho_t-m\,, \\
\partial_t A_t+ j_t =0\,.
\end{array}
\right.
\end{equation}
One of the advantages of the introduction of the vector variables $A$
is that by \eqref{Arj}, the continuity equation is automatically
satisfied.  Then the dynamical contribution \eqref{rfhdyn} to the
hydrodynamical functional can be rewritten using only the vector $A$
as
\begin{equation}\label{actioncurr}
  \mc I_{\textrm{dyn}}
  (A) = \frac 14 \int_0^T\!dt\int\!dx\,
  \frac{\big| \partial_t A - \tilde D_\mathrm{h}(\nabla\cdot A) \nabla
    \left(\nabla\cdot A\right)\big|^2}{\tilde\sigma(\nabla\cdot A)},
\end{equation}
where $\tilde D_\mathrm{h}(\rho):=D_\mathrm{h}(\rho+m)$ and
$\tilde\sigma(\rho):=\sigma(\rho+m)$.
Formula \eqref{actioncurr} can interpreted as an action functional
obtained as the integration of a Lagrangian, i.e.\ we have $\mc
I_{\textrm{dyn}}(A)=\int_0^Tdt\int dx \,\mathcal L(\partial_t A, A)$
for the Lagrangian function $\mathcal L$ that depends just on
$\partial_t A$ and $A$ that is the integrated term in
\eqref{actioncurr}.
The critical points of the functional $ \mc I_{\textrm{dyn}}$ can
therefore be identified by the machinery of the
Lagrangian and Hamiltonian formalism.
This is the advantage of the introduction of the vector $A$ in
\eqref{defA}.  Note that $\mc I_{\textrm{dyn}}$ is invariant under the
gauge transformation $A\mapsto A+A'$ when $A'$ is time independent and
$\nabla \cdot A'=0$.

\smallskip
The variational problem (HSPDC) can be faced using the variables $A$ as follows. Introduce the function
\begin{equation}\label{HA}
\hat\psi_T(\mathcal A_0,\mathcal A_1):=\frac 1T \inf \Big\{ \mc I_{\textrm{dyn}} (A),\; A_0=\mathcal A_0,\:A_T=\mathcal A_1\Big\}
\end{equation}
By the gauge invariance we have that 
\begin{equation}
\hat\psi_T(\mathcal A_0,\mathcal A_1)=\hat\psi_T(\mathcal A_0+A',\mathcal A_1+A')\,, \qquad \forall A'\ \  \textrm{such that}\ \ \nabla \cdot A'=0\,.
\end{equation}
 Moreover if $\rho_0, \rho_1$ are the densities of $\mu_0, \mu_1$, we have 
\begin{equation}
\psi_T(\mu_0,\mu_1; \bar J)=\hat\psi_T(\mathcal A_0,\mathcal A_1)\,, \qquad \textrm{if}\ \  \nabla\cdot \mathcal A_0=\rho_0-m\,,\  \mathcal A_0-\mathcal A_1=T\bar J\,.
\end{equation}

The variational problem \eqref{HA} is the classic minimization problem
for the action $\mc I_{\textrm{dyn}}$ under a pinning condition at the
extrema. The solution is obtained either by Euler Lagrange equations
or Hamilton equations.

The Hamiltonian associated to the action in \eqref{actioncurr} reads
\cite{BDGJLrev}
\begin{equation}\label{HamA}
  \mc H (A,B) = \int\!dx\,\big[
  \tilde\sigma(\nabla\cdot A) |B|^2 - \nabla\cdot A \, \nabla\cdot
  \big( \tilde D_\mathrm{h}(\nabla\cdot A) B \big)\big],
\end{equation}
where $B\colon \bb T^d\to \bb R^d$ is the momentum conjugated to $A$. We point out that this Hamiltonian on vector variables, extends the Hamiltonian on scalar variables \eqref{ham}. In particular if
$\nabla\cdot A=\rho-m$ and $B=-\nabla H$, we have $\mc H (A,B)=\ms H(\rho, H)$. 
The Hamiltonian equations associated to \eqref{HamA} are
\begin{equation}
  \label{excaneq}
  \begin{cases}
    \partial_t A = \tilde D_\mathrm{h}(\nabla\cdot A)\nabla \nabla\cdot A
    + 2 \tilde\sigma(\nabla\cdot A) B,\\
    \partial_t B = - \nabla\Big[
    \left(\tilde D_\mathrm{h}(\nabla\cdot A) \nabla\right)\cdot B
    - 2 \tilde \sigma'(\nabla\cdot A) |B|^2 \Big]\,.
  \end{cases}
\end{equation}
We can fix the gauge freedom determined by $A'$ for example in such a way that $A_0$ is of gradient type and
satisfy the initial condition $\nabla\cdot A_0=\rho_0-m$. This can be done defining $A_0:=\nabla a_0$ where $a_0$ is a solution of $\Delta a_0=\rho_0-m$.

Consider the pair $(\pi^*,\ms J^*)$ with density and current representation $(\rho^*,j^*)$ that is a minimizer for 
(HSPDC). In particular we therefore have that $\rho^*_0,\rho^*_T$ are the density of respectively $\mu_0,\mu_1$ and $\int_0^Tj^*_s\!ds=T\bar J$.  We define then 
\begin{equation}
\begin{cases}
A^*_t = A^*_0 - \int_0^t\!ds\, j^*_s\,,\\
B^*_t=\frac{\partial \mathcal L(\partial_tA^*_t,A^*_t)}{\partial (\partial_t A_t)}=-\frac{j^*_t+D_\mathrm{h}(\rho_t^*)\nabla\rho^*_t}{2\sigma(\rho_t^*)}\,,
\end{cases}
\end{equation}
and observe that by the constraints of (HSPDC) we have $A^*_T-A^*_0=T\bar J$ and $\nabla\cdot A^*_T=\rho^*_T-m$. We have that $(A^*,B^*)$ satisfy
the Hamilton equations \eqref{excaneq}; this follows by the general theory, since the critical points of the problem \eqref{HA} are solutions of the Hamiltonian equations \eqref{excaneq} (see \cite{BDGJLrev} for a brief discussion).
\smallskip

We point out that the vectorial Hamiltonian structure given by \eqref{HamA} and \eqref{excaneq} extends and includes the scalar Hamiltonian structure given by \eqref{ham}, \eqref{caneq}.
As can be verified by a direct computation , given $(\rho,H)$ a solution of \eqref{caneq} then if we compute
as before $A_0=\nabla a_0$ with $\Delta a_0=\rho_0-m$ and then define
\begin{equation}\label{ABrg}
\begin{cases}
A_t = A_0 - \int_0^t\!ds\,\big(-D_\mathrm{h}(\rho_s)\nabla\rho_s+2\sigma(\rho_s)\nabla H_s\big)\,,\\
B_t=-\nabla H_t\,,
\end{cases}
\end{equation}
we obtain a solution of \eqref{excaneq}. In \eqref{excaneq} there are however solutions
for which $B$ is not gradient and that cannot be written starting from solutions to \eqref{caneq} as in \eqref{ABrg}.

The proof of following statement, that characterizes the minimizers to
(MSPC) in terms to the minimizers to (HSPC), and minimizers of (MSPDC)
in terms of minimizers of (HSPDC) is achieved by the same arguments in
Theorem~\ref{t:dspvsmsp} and it is therefore omitted.

\begin{theorem}
	\label{t:hspcmspc}
	For each $T>0$ and $\bar J, \mu_0, \mu_1$ we have
\begin{equation*}
	\phi_T(\bar J)=\Phi_T(\bar J), \qquad 
	\psi_T(\mu_0,\mu_1,\bar J)=\Psi_T(\mu_0,\mu_1,\bar J).
\end{equation*}        
	Let $(\pi^*,\ms J^*)$ be a minimizer for (HSPC), respectively (HSPCD) and denote by
	$(\rho^*,j^*)= (\rho^*_t,j^*_t)_{t\in[0,T]}$ the corresponding
	density and current representative. 
	Then a minimizer for (MSPC) respectively (MSPDC) is provided by $R^*$, the law of the
	time-inhomogeneous Markov diffusion with initial condition
	$R^*_0 =\rho^*_0\,dx$ and time dependent generator
	\begin{equation*}
	L^*_t = \nabla\cdot D_\mathrm{s} (\rho^*_t) \nabla 
	+
	\Big\{
	\big[ D_\mathrm{s} (\rho^*_t) -  D_\mathrm{h} (\rho^*_t)\big]
	\frac {\nabla \rho^*_t}{\rho^*_t} 
	+ 2\frac {\sigma(\rho^*_t)}{\rho^*_t} E^*_t \Big\}\cdot \nabla
	\end{equation*}
	where $j^*= - D_\mathrm{h} (\rho^*)\nabla\rho^*
	+2\sigma(\rho^*)E^*$.\hfill\break 
	Conversely, if $R^*$ is a minimizer for (MSPC) respectively (MSPDC ) then $R^*$ is Markovian $R^*=P(R^*)$ and
	$\big( (R^*_t)_{t\in[0,T]}, \int\!R^*(dX)\, \mc J(X)\big)$ is a
	minimizer for (HSPC) respectively (HSPDC). \hfill\break
	The pair $(A^*,E^*)$ defined as in \eqref{ABrg} solves the Hamilton equations \eqref{excaneq}.
\end{theorem}

In Theorem~\ref{t:hspcmspc} we have characterized the minimizers to
(MSPC) and (MSPDC) as forward Markov measures on $C([0,T];\bb
R^d)$. We can also characterize them as \emph{backward} Markov
measures. This is the content of the next statement whose proof is
similar that of Lemma~\ref{t:trR*} and it is therefore omitted.

\begin{lemma}
	\label{t:trR*c}
	Let $\theta$ be the involution on $C([0,T];\bb T^d)$ defined by
	$(\theta X)_t=X_{T-t}$.
	Given a minimizer $R^*$ for \emph{(MSPC)} or \emph{(MSPDC)};  set $\hat R:= R^*\circ \theta$ and
	let $\hat E\colon [0,T]\times \bb T^d\to \bb R^d$ be such that
	\begin{equation}\label{bH}
	E_t^*+\hat E_{t} = \nabla f'(\rho^*_t), \qquad t\in[0,T]
	\end{equation}
	where $\rho^*$ is the density of the single time marginals of $R^*$
	while $E^*_t$ is the conjugated momentum.
	Then $\hat R$ is the law of the time-inhomogeneous Markov diffusion
	with time dependent generator
	\begin{equation}
	\label{tdgpopr2}
	\hat L_t = \nabla\cdot D_\mathrm{s} (\rho^*_{T-t}) \nabla 
	+
	\Big\{
	\big[ D_\mathrm{s} (\rho^*_{T-t}) -  D_\mathrm{h} (\rho^*_{T-t})\big]
	\frac {\nabla \rho^*_{T-t}}{\rho^*_{T-t}} 
	+ 2 \frac {\sigma(\rho^*_{T-t})}{\rho^*_{T-t}}
	\hat E_{T-t} \Big\}\cdot \nabla
	\end{equation}
	and initial condition $\mu_1$. 
\end{lemma}
Notice that formula \eqref{bH} is a generalization of the Born's
formula \eqref{Born}, \eqref{veraborn} and is a special case of the
fluctuation-dissipation relation in  \cite[Eq.~(5.5)]{BDGJLptcur}.

By the same argument used for equations \eqref{secH}, the pair
$(E^*,\hat E)$ solves
\begin{equation*}
\begin{cases}
\partial_t E^* = - \nabla\Big[\left(
D_\mathrm{h}(\rho^*) \nabla\right)\cdot E^*
+2 \sigma'(\rho^*) |E^*|^2 \Big]\,,\\
\partial_t \hat E = \nabla\Big[\left(
D_\mathrm{h}(\rho^*) \nabla\right)\cdot \hat E
+ 2 \sigma'(\rho^*) |\hat E|^2 \Big]
\end{cases}
\end{equation*}
An interesting issue is whether there is also for the Hamilton
equations \eqref{excaneq} a relation with integrable systems as
happens for the projected Hamiltonian equations \eqref{caneq}.

\appendix

\section{Relative Entropy Representation
of the QRV Functional \textbf{(by R.\ Chetrite)}}
\label{s:rap}

In this appendix, we show that the QRV functional (see equation   (\ref{rfqrv})
of the main text) governing the large deviations of the empirical
measure admits, in a general setting, a simplified representation
in terms of relative entropy.

\subsection*{Generalized set-up considered in the appendix}

We consider here a general framework encompassing and extending the
main settings discussed in the main text, namely stochastic lattice
gases and mean-field interacting Brownian motions, as well as other
interacting particle systems such as short-range interacting Brownian
motions. 

We assume that the hydrodynamic limit of the empirical density is
governed by the continuity equation 
\begin{equation*}
\partial_{t}\rho_{t}+\nabla\cdot J_{h,\rho_{t}}=0,
\end{equation*}
where the vector $J_{h,\rho_{t}}$ denotes the hydrodynamic current,
whose expression depends on the specific model. For instance, it is
given by the equation (\ref{phe}) of the main text for lattice gases
and by (\ref{phemf}) for mean-field interacting Brownian motions. In
what follows, however, we work in full generality and do not assume
any specific form of $J_{h,\rho_{t}}$.

Our starting point is the QRV functional established by Quastel,
Rezakhanlou and Varadhan in \cite{QRV} for the large deviations of the
empirical measure (see (\ref{rfqrv}) with (\ref{rfh}) in the main
text), which can be written as
\begin{align}
\mathfrak{F}\left[R\right]=I_\mathrm{in}\left[\rho_{0}^{R}\right]+I_\mathrm{dyn}\left[\rho^{R},j^{R}\right]+\Ent\left[\left.R\right|P\left(R\right)\right].\label{eq:QRV-1}  
\end{align}
Here $(\rho_{t}^{R},j_{t}^{R})$ denotes the instantaneous density and
current\footnote{Formally, they are defined by
\begin{equation*}
\text{\ensuremath{\begin{cases}
 \rho_{t}^{R}(x):=\mathbb{E}_{R}\left(\delta\left(X_{t}-x\right)\right)\\
 j_{t}^{R}(x):=\mathbb{E}_{R}\left(\delta\left(X_{t}-x\right)\circ dX_{t}\right) 
\end{cases}}, and they satisfy the continuity equation \ensuremath{\partial_{t}\rho_{t}^{R}}+\ensuremath{\nabla\cdot j_{t}^{R}}=0}.\label{eq:pho-cur}
\end{equation*}
} associated with the (possibly non-Markovian) path measure $R\in\mathcal{M}_{+}\left(D\left[0,T\right];\mathbb{T}^{d}\right)$.

The dynamical contribution to the rate functional is given by
\begin{equation}
I_\mathrm{dyn}\left[\rho,j\right]=\frac{1}{4}\int_{0}^{T}dt\int dx\left(j_{t}(x)-J_{h,\rho_{t}}(x)\right)\cdot\left(\sigma\left[\rho_{t}(x)\right]\right)^{-1}\left(j_{t}-J_{h,\rho_{t}}(x)\right),\label{eq:2,5-1}
\end{equation}
 if $\ensuremath{\partial_{t}\rho_{t}+\nabla\cdot j_{t}=0}$, and
$I_{dyn}\left[\rho,j\right]=\infty$ otherwise.
This generalizes and unifies   (\ref{phe}, \ref{phe-I})  
and (\ref{ldmf}) in the main text. Here $\sigma\left[\rho\right]$ denotes
the (symmetric positive definite) mobility matrix. 

Finally  $P\left(R\right)$ denotes the Markovian
path measure having the same instantaneous density and current 
$(\rho_{t}^{R},j_{t}^{R})$ as the path measure $R$. Its generator
is given by the time-inhomogeneous
\begin{equation}
L_{t}^{R}=\nabla\cdot D_{\mathrm{s}}\left(\rho_{t}^{R}\right)\nabla+\frac{D_{\mathrm{s}}\left(\rho_{t}^{R}\right)\nabla\rho_{t}^{R}+j_{t}^{R}}{\rho_{t}^{R}}\cdot\nabla.\label{eq:Lstar-2}
\end{equation}
which generalizes the equation (\ref{tdgp}) in the main text. 
Here the symmetric
positive definite matrix\footnote{
In contrast with the assumptions made in the main text, in this appendix
we do not require that $\sigma\left[\rho\right]$ and $D_{\mathrm{s}}\left(\rho\right)$
be proportional to identity} $D_{\mathrm{s}}\left(\rho\right)$ denotes the self-diffusion matrix, encoding
the asymptotic variance of a tagged particle. 

The precise domain of validity of the large deviation principle  (\ref{eq:QRV-1})
associated with the QRV functional is not fully characterized in the
literature. Nevertheless, it is widely believed to hold in a broad,
possibly universal range of settings, as emphasized by Quastel, Rezakhanlou,
and Varadhan in the abstract of \cite{QRV}
where they state : “We prove a large deviation principle, with a rate
function that is more or less universal.”

However, there exist interacting particle systems for which such a
large deviation principle remains an open problem. This is in particular
the case of Ginzburg--Landau dynamics \cite{Guo}
or in the Kipnis-Marchioro-Presutti model \cite{Presutti}, where
the notion of a tagged particle is not a priori available, and consequently
the QRV-type structure cannot be directly formulated.

\subsection*{Main result: relative entropy representation of the QRV functional}
The main result of this appendix is the following
representation. Within the general framework introduced above, the
large deviation rate function (\ref{eq:QRV-1}) admits the following
simplified representation
\begin{equation}
\mathfrak{F}\left[R\right]=I_\mathrm{in}\left[\rho_{0}^{R}\right]+\Ent\left[\left.R\right|Q\left(\rho^{R},j^{R}\right)\right],\label{eq:QRV-simplified}
\end{equation}
 where $(\rho_{t}^{R},j_{t}^{R})$ denotes the density--current pair
(\ref{eq:pho-cur}) associated with the (possibly non-Markovian) path
measure $R$, and $Q\left(\rho,j\right)$ is the law of a time-inhomogeneous
Markov diffusion with initial distribution given by the positive measure
$\rho_{0}$ and generator given by 
\begin{equation}
L_{t}^{\rho,j}:=\nabla\cdot D_{\mathrm{s}}\left(\rho_{t}\right)\nabla+\frac{\left(\begin{array}{c}
D_{\mathrm{s}}\left(\rho_{t}\right)\left(\nabla\rho_{t}\right)+\left(\rho_{t}D_{\mathrm{s}}\left(\rho_{t}\right)\right)^{\frac{1}{2}}\left(\sigma\left(\rho_{t}\right)\right)^{-\frac{1}{2}}J_{h,\rho_{t}}\\
+\left(Id-\left(\rho_{t}D_{\mathrm{s}}\left(\rho_{t}\right)\right)^{\frac{1}{2}}\left(\sigma\left(\rho_{t}\right)\right)^{-\frac{1}{2}}\right)j_{t}
\end{array}\right)}{\rho_{t}}\cdot\nabla.\label{eq:Q-final}
\end{equation}
In this relation, $A^{\frac{1}{2}}$ denotes the symmetric positive
definite square root of a positive definite matrix $A$. This result
reveals an underlying structural identity behind the QRV functional
and extends the entropy representation obtained in the case of mean-field
interacting Brownian motions (see reference \cite{CL}) by allowing
a dependence of the reference measure $Q\left(\rho^{R},j^{R}\right)$
on the current $j^{R}.$

At this stage, it is natural to single out the class of models for
which the reference measure $Q\left(\rho^{R},j^{R}\right)$ does not
depend on the current $j^{R}$ . This occurs precisely when the matrix
identity
\[
\rho D_{\mathrm{s}}\left(\rho\right)=\sigma\left(\rho\right),
\]
holds. Under this condition, the generator (\ref{eq:Q-final})
reduces to 
\begin{equation}
L_{t}^{\rho,j}=L_{t}^{\rho}:=\nabla\cdot D_{\mathrm{s}}\left(\rho_{t}\right)\nabla+\left(\frac{D_{\mathrm{s}}\left(\rho_{t}\right)\nabla\rho_{t}+J_{h,\rho_{t}}}{\rho_{t}}\right)\cdot\nabla.\label{eq:Q-final-4}
\end{equation}
This condition is satisfied, for instance, in the zero-range process
and in the case of mean-field interacting Brownian motions, as we
show in the next section.

For future work, it would be interesting to determine whether a direct
derivation of the relative entropy representation (\ref{eq:QRV-simplified})
can be obtained more simply than the original proof of the QRV functional
(\ref{eq:QRV-1}) established in \cite{QRV}.

\subsection*{Examples of interacting particle systems}

We illustrate the above representation through several classes of
interacting particle systems.

\begin{enumerate}
\item \textbf{Spatially homogeneous zero-range process}: In this case, the
hydrodynamic current and transport coefficients are given by \cite{JLS,KL}
\[
J_{h,\rho}=-\phi'\left(\rho\right)\nabla\rho,\;\sigma\left(\rho\right)=\phi\left(\rho\right)Id,\;D_{\mathrm{s}}\left(\rho\right)=\frac{\phi\left(\rho\right)}{\rho}Id.
\]
A direct computation shows that the generator
(\ref{eq:Q-final}) reduces to (\ref{eq:Q-final-4}) and then to the
equation   (\ref{tdgzr}) in the main text.

\item \textbf{Mean-field interacting diffusion with one body potential $V$
and two body potential $U$ }: The hydrodynamic current is given by
\cite{DG}
\[
J_{h,\rho}=-D\nabla\rho-\left(\nabla V\right)\rho-\rho\left(\rho\ast\nabla U\right),
\]
 where $D_{}$ is the space homogeneous diffusion matrix of each
individual diffusion process, and then the transport coefficients
satisfy
\[
\sigma\left(\rho\right)=\rho D,\;\;\;D_{\mathrm{s}}\left(\rho\right)=D.
\]
In this case, the generator (\ref{eq:Q-final}) simplifies to 
\begin{equation}
L_{t}^{\rho,j}=L_{t}^{\rho}:=\nabla\cdot D\nabla-\left(\left(\nabla V\right)+\left(\rho_{t}\ast\nabla U\right)\right)\cdot\nabla.\label{eq:Q-final-2}
\end{equation}

\item \textbf{Short range interacting Brownian motions for $d\geq2$} (see
references \cite{V,osada1,osada2}): the hydrodynamic current
and the mobility matrix are given by
\[
J_{h,\rho}=-\frac{1}{2}S''\left(\rho\right)\nabla\rho,\;\;\;\sigma\left(\rho\right)=\rho,
\]
where $S\left(\rho\right)$ is the thermodynamic entropy. The generator
(\ref{eq:Q-final}) then takes the form
\begin{equation}
L_{t}^{\rho,j}=\nabla\cdot D_{\mathrm{s}}\left(\rho_{t}\right)\nabla+\frac{\left(\begin{array}{c}
D_{\mathrm{s}}\left(\rho_{t}\right)\left(\nabla\rho_{t}\right)-\frac{1}{2}\sqrt{D_{\mathrm{s}}\left(\rho_{t}\right)}S''\left(\rho_{t}\right)\nabla\rho_{t}\\
+\left(1-\sqrt{D_{\mathrm{s}}\left(\rho_{t}\right)}\right)j_{t}
\end{array}\right)}{\rho_{t}}\cdot\nabla.\label{eq:Q-final-1}
\end{equation}
In contrast with the previous cases, this generator retains an explicit
dependence on the current. 

\item \textbf{One-dimensional short-range interacting Brownian motions with
partial reflection} (See reference \cite{Seo}): In this case,
\[
J_{h,\rho}=-\nabla\rho,\;\;\;\sigma\left(\rho\right)=\rho,\;\;\;D_{\mathrm{s}}\left(\rho\right)=\frac{\lambda}{\lambda+\rho},
\]
The generator (\ref{eq:Q-final}) then becomes
\begin{equation}
L_{t}^{\rho,j}=\nabla\cdot\frac{\lambda}{\lambda+\rho_{t}}\nabla+\left(1-\sqrt{\frac{\lambda}{\lambda+\rho_{t}}}\right)\frac{-\sqrt{\frac{\lambda}{\lambda+\rho_{t}}}\left(\nabla\rho_{t}\right)+j_{t}}{\rho_{t}}\cdot\nabla.\label{eq:Q-final-3}
\end{equation}
In this case as well, this generator retains an explicit dependence
on the current. 

\end{enumerate}

\subsection*{Proof of relative entropy representation
  (\ref{eq:QRV-simplified}) with generator (\ref{eq:Q-final})}

We start from the QRV representation (\ref{eq:QRV-1}).
We first rewrite the relative entropy term in this equation by using
the chain rule (analogue of (\ref{chain}) in the main text)
\begin{align*}
  \Ent\left[\left.R\right|P\left(R\right)\right] &
  =\Ent\left[\left.R\right|Q\left(\rho^{R},j^{R}\right)\right] -\int dR\,
  \ln\left(\frac{dP\left(R\right)}{dQ\left(\rho^{R},j^{R}\right)}\right). 
\end{align*}
Substituting into (\ref{eq:QRV-1}), we obtain 
\begin{align}
\mathfrak{F}\left[R\right] & =I_\mathrm{in}\left[\rho_{0}^{R}\right]+\Ent\left[\left.R\right|Q\left(\rho^{R},j^{R}\right)\right]+I_\mathrm{dyn}\left[\rho^{R},j^{R}\right]-\int dR\ln\left(\frac{dP\left(R\right)}{dQ\left(\rho^{R},j^{R}\right)}\right).\label{eq:QRV-modified}
\end{align}

We now introduce a family of reference measures $\mathbb{Q}\left(\rho,j\right)$,
defined as the law of time-inhomogeneous Markov diffusion with generator
\begin{equation}
L_{t}^{\rho,j}:=\nabla\cdot D_{\mathrm{s}}\left(\rho_{t}\right)\nabla+\left(\frac{D_{\mathrm{s}}\left(\rho_{t}\right)\nabla\rho_{t}+M\left(\rho_{t}\right)J_{h,\rho_{t}}+\left(Id-M\left(\rho_{t}\right)\right)j_{t}}{\rho_{t}}\right)\cdot\nabla,\label{eq:Q}
\end{equation}
where $M\left(\rho\right)$ is a matrix-valued function to be specified
later, which allows that $L_{t}^{\rho^{R},j^{R}}$ is a perturbation
of (\ref{eq:Lstar-2}) $L_{t}^{R}$ (recovered for $M=0$).

We now compute the Radon--Nikodym derivative between $P\left(R\right)$
and $Q\left(\rho^{R},j^{R}\right).$ Using generators (\ref{eq:Lstar-2})
and (\ref{eq:Q}), by Girsanov’s formula \cite{Stroock,Chetrite},
we obtain 
\begin{align}
\ln\left(\frac{dP\left(R\right)}{dQ\left(\rho^{R},j^{R}\right)}\right)\left[X\right] & =\int_{0}^{T}A_{t}(X_{t})\cdot d_{t}^{-1}(X_{t})\circ dX_{t}\label{eq:GS-1}\\
 & +\int_{0}^{T}\left(-d_{t}^{-1}A_{t}\cdot\left(F_{t}^{\rho^{R},j^{R}}+\frac{A_{t}}{2}\right)-\frac{1}{2}\nabla\cdot A_{t}\right)(X_{t})dt,
\end{align}
with 
 \begin{equation}
d_{t}:=2D_{\mathrm{s}}\left(\rho_{t}^{R}\right),\;\;\;F_{t}^{\rho^{R},j^{R}}:=\frac{D_{\mathrm{s}}\left(\rho_{t}^{R}\right)\nabla\rho_{t}^{R}+M\left(\rho_{t}^{R}\right)J_{h,\rho_{t}^{R}}+\left(Id-M\left(\rho_{t}^{R}\right)\right)j_{t}^{R}}{\rho_{t}^{R}},\label{eq:fat}
\end{equation}
 and 
\begin{equation}
A_{t}:=\text{M\ensuremath{\left(\rho_{t}^{R}\right)\frac{j_{t}^{R}-J_{h,\rho_{t}^{R}}}{\rho_{t}^{R}}}}.
\label{eq:fat2}
\end{equation}

Using that $\rho^{R}$ and $j^{R}$ represent respectively the marginal
density and current under $R$, respectively, we obtain, after an
integration by parts,
\begin{align}
\int dR\ln\left(\frac{dP\left(R\right)}{dQ\left(\rho^{R},j^{R}\right)}\right) & =\int_{0}^{T}dt\int dx\left[\begin{array}{c}
A_{t}(x)\cdot d_{t}^{-1}(x)j_{t}^{R}(x)\\
-d_{t}^{-1}A_{t}\cdot\left(F_{t}^{\rho^{R},j^{R}}+\frac{A_{t}}{2}\right)\rho_{t}^{R}(x)\\
-\frac{1}{2}\left(\nabla\cdot A_{t}\right)(x)\rho_{t}^{R}(x)
\end{array}\right]\nonumber \\
 & =\int_{0}^{T}dt\int dxA_{t}(x)\cdot\left[\begin{array}{c}
d_{t}^{-1}(x)j_{t}^{R}(x)\\
-d_{t}^{-1}(x)\left(F_{t}^{\rho^{R},j^{R}}+\frac{A_{t}}{2}\right)\rho_{t}^{R}(x)\\
+\frac{1}{2}\nabla\rho_{t}^{R}(x)
\end{array}\right].\label{eq:Girs-spec}
\end{align}

A direct computation with relations (\ref{eq:fat},\ref{eq:fat2}) yields
\begin{equation}
d_{t}^{-1}\cdot\left(j_{t}^{R}-\left(F_{t}^{\rho^{R},j^{R}}+\frac{A_{t}}{2}\right)\rho_{t}^{R}\right)+\frac{1}{2}\nabla\rho_{t}^{R}=\text{\ensuremath{\frac{1}{4}\left(D_{\mathrm{s}}\left(\rho_{t}^{R}\right)\right)^{-1}}M\ensuremath{\left(\rho_{t}^{R}\right)\left(j_{t}^{R}-J_{h,\rho_{t}^{R}}\right)}.}\label{eq:marre}
\end{equation}

Combining the relations
(\ref{eq:fat},\ref{eq:fat2},\ref{eq:Girs-spec},\ref{eq:marre}) we obtain 
{\footnotesize
\begin{align*}
\int dR\ln\left(\frac{dP\left(R\right)}{dQ\left(\rho^{R},j^{R}\right)}\right)\\
=\frac{1}{4}\int_{0}^{T}dt\int dx\left(j_{t}^{R}(x)-J_{h,\rho_{t}^{R}}(x)\right)\cdot\left[\left(M\left(\rho_{t}^{R}\right)\right)^{T}\left(\rho_{t}^{R}D_{\mathrm{s}}\left(\rho_{t}^{R}\right)\right)^{-1}M\left(\rho_{t}^{R}\right)\right]\left(j_{t}^{R}(x)-J_{h,\rho_{t}^{R}}(x)\right).
\end{align*}}

Combining this expression with the definition (\ref{eq:2,5-1}) of
$I_\mathrm{dyn}$, we obtain
\begin{align*}
I_\mathrm{dyn}\left[\rho^{R},j^{R}\right]-\int dR\ln\left(\frac{dP\left(R\right)}{dQ\left(\rho^{R},j^{R}\right)}\right)\\
=\frac{1}{4}\int_{0}^{T}dt\int dx\left(j_{t}^{R}-J_{h,\rho_{t}^{R}}\right)\cdot\left[\begin{array}{c}
\left(\sigma\left(\rho_{t}^{R}\right)\right)^{-1}\\
-\left(M\left(\rho_{t}^{R}\right)\right)^{T}\left(\rho_{t}^{R}D_{\mathrm{s}}\left(\rho_{t}^{R}\right)\right)^{-1}M\left(\rho_{t}^{R}\right)
\end{array}\right]\left(j_{t}^{R}-J_{h,\rho_{t}^{R}}\right).
\end{align*}

We now choose $M\left(\rho\right)=M^{\ast}\left(\rho\right)$ so that $\left(M^{\ast}\left(\rho\right)\right)^{T}\left(\rho D_{\mathrm{s}}\left(\rho\right)\right)^{-1}M^{\ast}\left(\rho\right)=\left(\sigma\left(\rho\right)\right)^{-1}.$

This is achieved by taking 

\begin{equation}
M^{\ast}\left(\rho\right)=\left(\rho D_{\mathrm{s}}\left(\rho\right)\right)^{\frac{1}{2}}\left(\sigma\left(\rho\right)\right)^{-\frac{1}{2}}.\label{eq:Mstar}
\end{equation}

With this choice, the above difference vanishes, and we obtain\footnote{Note also that since $R$ and $P\left(R\right)$ have the same density--current
pair \textcolor{black}{$(\rho_{t}^{R},j_{t}^{R})$}\textcolor{red}{{}
}density-current, we also have the identity
\[
\int dR\ln\left(\frac{dP\left(R\right)}{dQ\left(\rho^{R},j^{R}\right)}\right)=\int dP\left(R\right)\ln\left(\frac{dP\left(R\right)}{dQ\left(\rho^{R},j^{R}\right)}\right)=\Ent\left(\left.P\left(R\right)\right|Q\left(\rho^{R},j^{R}\right)\right).
\]
}
\[
I_\mathrm{dyn}\left[\rho^{R},j^{R}\right]=\int dR\ln\left(\frac{dP\left(R\right)}{dQ\left(\rho^{R},j^{R}\right)}\right).
\]
With (\ref{eq:QRV-modified}), this concludes the proof because (\ref{eq:QRV-modified})
reduces to the desired representation (\ref{eq:QRV-simplified}).
Moreover the mesure $Q\left(\rho,j\right)$ is then the law of the
time-inhomogeneous diffusion with generator (\ref{eq:Q}) with the
choice (\ref{eq:Mstar}): 
\begin{equation}
L_{t}^{\rho,j}:=\nabla\cdot D_{\mathrm{s}}\left(\rho_{t}\right)\nabla+\frac{\left(\begin{array}{c}
D_{\mathrm{s}}\left(\rho_{t}\right)\left(\nabla\rho_{t}\right)+\left(\rho_{t}D_{\mathrm{s}}\left(\rho_{t}\right)\right)^{\frac{1}{2}}\left(\sigma\left(\rho_{t}\right)\right)^{-\frac{1}{2}}J_{h,\rho_{t}}\\
+\left(Id-\left(\rho_{t}D_{\mathrm{s}}\left(\rho_{t}\right)\right)^{\frac{1}{2}}\left(\sigma\left(\rho_{t}\right)\right)^{-\frac{1}{2}}\right)j_{t}
\end{array}\right)}{\rho_{t}}\cdot\nabla,\label{eq:Q-final-5}
\end{equation}
 which is precisely the announced generator (\ref{eq:Q-final}).

\section*{Acknowledgments}
We thank C. Leonard for useful discussions.
L.B. has been co-funded by the European Union (ERC CoG KiLiM,
project number 101125162).
D.G. acknowledges partial financial support from the Italian Research Funding
Agency (MIUR) through PRIN project “Emergence of condensation-like
phenomena in interacting particle systems: kinetic and lattice
models”, grant no. 202277WX43,  and  by the Italian Ministry of
Foreign Affairs and International Cooperation, grant number BR26GR05.   
R.C. is supported by the project RETENU ANR-20-CE40 of the French National Research Agency (ANR).

\end{document}